\documentclass{preprint}
\usepackage[english]{babel}
\usepackage{amsmath,amssymb,xcolor, tikz, graphicx, stmaryrd}
\usepackage{newtxtext}
\usepackage{makecell, booktabs}

\usepackage{mhequ,mhenvs}

\usepackage{xifthen}

\usepackage{breakurl}
\usepackage{hyperref}

\newcommand{\RR}{\mathbf{R}}
\newcommand{\NN}{\mathbf{N}}
\newcommand{\ZZ}{\mathbf{Z}}

\newcommand{\TT}{\mathbf{T}}

\newcommand{\bM}{\mathbf{M}}

\newcommand{\EE}{\mathbb{E}}

\newcommand{\PP}{\mathbb{P}}

\newcommand{\mL}{\mathcal{L}}
\newcommand{\mN}{\mathcal{N}}

\newcommand{\mC}{\mathcal{C}}

\newcommand{\mM}{\mathcal{M}}

\newcommand{\mF}{\mathcal{F}}

\newcommand{\mX}{\mathcal{X}}

\newcommand{\mS}{\mathcal{S}}
\newcommand{\mB}{\mathcal{B}}
\newcommand{\mA}{\mathcal{A}}
\newcommand{\mH}{\mathcal{H}}

\newcommand{\mf}[1]{\mathfrak{#1}}

\newcommand{\para}{\varolessthan}
\newcommand{\rpara}{\varogreaterthan}
\newcommand{\reso}{\varodot}

\newcommand{\ve}{\varepsilon}

\newcommand{\vt}{\vartheta}

\newcommand*{\ud}{\mathrm{\,d}}
 
\renewcommand{\div}{\operatorname{div}}

\makeatletter
\newcommand{\opnorm}{\@ifstar\@opnorms\@opnorm}
\newcommand{\@opnorms}[1]{%
  \left|\mkern-1.5mu\left|\mkern-1.5mu\left|
   #1
  \right|\mkern-1.5mu\right|\mkern-1.5mu\right|
}
\newcommand{\@opnorm}[2][]{%
  \mathopen{#1|\mkern-1.5mu#1|\mkern-1.5mu#1|}
  #2
  \mathclose{#1|\mkern-1.5mu#1|\mkern-1.5mu#1|}
}
\makeatother

\usetikzlibrary{arrows.meta,calc,decorations.markings,math,patterns, shapes}
\tikzstyle{terminator} = [rectangle, draw, text centered, rounded corners, minimum height=2em]
\tikzstyle{process} = [rectangle, draw, text centered, minimum height=2em]
\tikzstyle{decision} = [diamond, draw, text centered, minimum height=2em]
\tikzstyle{data}=[trapezium, draw, text centered, trapezium left angle=60, trapezium right angle=120, minimum height=2em]
\tikzstyle{connector}  = [draw, very thick]
\tikzstyle{connectorG} = [draw, green!70!gray, very thick]
\tikzstyle{arrow} = [very thick,->,>=stealth]

\def\CN{\mathcal{N}}

\let\f\frac
\def\tf#1#2{{\textstyle{\frac{#1}{#2}}}}

\newtheorem{assumption}{Assumption}[section]

\DeclareSymbolFont{timesoperators}{T1}{ptm}{m}{n}
\SetSymbolFont{timesoperators}{bold}{T1}{ptm}{b}{n}

\makeatletter
\renewcommand{\operator@font}{\mathgroup\symtimesoperators}
\makeatother

\newcommand{\eqdef}{\stackrel{\mbox{\tiny\rm def}}{=}}

\def\fin{\mathrm{fin}}

\def\dash{\leavevmode\unskip\kern0.18em--\penalty\exhyphenpenalty\kern0.18em}
\def\slash{\leavevmode\unskip\kern0.15em/\penalty\exhyphenpenalty\kern0.15em}

\def\sotimes{\mathbin{\otimes_{\!s}}}
\let\f\frac
\def\ft#1#2{\textstyle{\frac{#1}{#2}}}
\def\nablas{\nabla_{\mathrm{sym}}}
\def\scal#1{\langle#1\rangle}

\begin{document}

\title{Global existence for perturbations of the 2D stochastic Navier--Stokes
equations with space-time white noise}

\author{Martin Hairer$^1 $ and Tommaso Rosati$^2$}
\institute{Imperial College London, UK, \email{m.hairer@imperial.ac.uk} 
   \and University of Warwick, UK, \email{t.rosati@warwick.ac.uk} }

\maketitle

\begin{abstract}
   We prove global in time well-posedness for perturbations of the 2D
   stochastic Navier--Stokes equations 
   \begin{equs}
          \partial_t u + u \cdot \nabla u & =  \Delta u - \nabla p + \zeta + \xi
          \;, \qquad u (0, \cdot) = u_{0} \;,\\
          \div (u) & =  0 \;,
   \end{equs}
   driven by additive space-time white noise $ \xi $, with perturbation $ \zeta
$ in the H\"older--Besov space $\mC^{-2 + 3\kappa} $, periodic boundary conditions and initial condition $
   u_{0} \in \mC^{-1 + \kappa} $ for any $ \kappa >0 $. The proof relies on an energy
   estimate which in turn builds on a dynamic high-low frequency decomposition and tools
   from paracontrolled calculus.
   Our argument uses
   that the solution to the linear equation is a $ \log$--correlated field,
   yielding a double exponential growth bound on the solution. Notably, our method does
   not rely on any explicit knowledge of the invariant measure to the SPDE,
   hence the perturbation $ \zeta $ is not restricted to the Cameron--Martin space of the
   noise, and the initial condition may be anticipative. Finally, we introduce
   a notion of weak solution that leads to well-posedness for all initial data $
   u_{0}$ in $ L^{2} $, the critical space of initial conditions.
\end{abstract}

\setcounter{tocdepth}{2}
\tableofcontents

\section{Introduction}
The aim of this article is to study the global in time well-posedness of the 2D
stochastic Navier-Stokes (SNS) equations
\begin{equs}\label{eqn-stochastic-ns-original}
   \partial_t u + u \cdot \nabla u  & =  \Delta u - \nabla p+ \Pi_{\times} (\zeta +
   \xi) \;, \qquad u (0, \cdot) = u_{0} (\cdot)\;, \\
    \div (u) & =  0 \;,
\end{equs}
for $ (t, x) \in [0, \infty) \times \TT^{2} $, with $ \TT^{2} $ the $ 2
$D torus and where \(\xi = (\xi_{i})_{i=1,2}\)
is a two-dimensional space-time white noise, namely a Gaussian generalised random
field which formally has the covariance
\[ \mathbb{E} [\xi_i (t, x) \xi_j (s, y)] = \delta_{t, x} (s, y) \delta_{i,j}\;. \]
The symbol $ \zeta $ denotes a perturbation belonging to $ \mC^{-2 + \kappa} $ (on parabolic space-time) and $
\Pi_{\times} $ is the projection on mean-free
functions $ \Pi_{\times} f = f - \int_{\TT^{2} } f(x) \ud x $ (introduced
   merely for simplicity, as the zero mode decouples from all others).
Our main result shows that there exists a null set $\CN$ such that for all
realisations of the noise $ \xi $ outside $\CN$, any initial condition
$ u_{0} \in L^{2} \cup \mC^{-1 + \kappa} $ and any perturbation $ \zeta \in
C([0, \infty]; \mC^{-2 + \kappa}) $, for arbitrary $ \kappa > 0 $, there exists
a unique solution to \eqref{eqn-stochastic-ns-original} \emph{for all times}.
For initial data in $ L^{2} $, this result requires the introduction of a suitable notion of weak
solution, while for initial data in $ \mC^{-1 + \kappa} $ we consider mild
solutions in the sense of Da Prato and Debussche.

Indeed, in a by now classical work, Da Prato and Debussche
\cite{DaPrato-Debusshe-2DNS-2002} establish the local
well-posedness of \eqref{eqn-stochastic-ns-original} (a similar approach
was taken earlier by Bourgain \cite{Bourgain1996} in a related context),  a first step in the
development of solution theories for singular SPDEs. 
Both the study of local
and global solutions to singular SPDEs has seen enormous progress in recent
years. In particular, with respect to global in time well-posedness we can
highlight at least three different lines of research.
On the one hand, a number of equations, including
\eqref{eqn-stochastic-ns-original} with $ \zeta = 0 $ but also
Burgers' equation and SQG equations, admit an
explicit invariant measure, in our case the Gaussian field
\begin{equ}
X = \int_{-\infty}^{0} P_{t -s} \mathbf{P} \Pi_{\times} \xi \ud s \;, 
\end{equ}
with $ P_{t} $ the
heat semigroup and $ \mathbf{P} $ the Leray projection. 
Explicit knowledge
of the invariant measure $\mu$ allows, under rather weak technical assumptions, 
to deduce global well-posedness for the equation
for $\mu$-almost all initial data \cite{DaPrato-Debusshe-2DNS-2002}, a result that
can be strengthened to all initial data, almost surely (with the null set possibly depending on
the initial condition), if the law of the solution
satisfies the strong Feller property, see for example
\cite[Theorem 4.10]{HairerStuart07SPDEs}
or \cite{zhu2017strong}. 
Approaches building 
on explicit knowledge of the invariant measure
appear to fail in our setting though since, as soon as $ \zeta \not\in
L^{2}_{\mathrm{loc}}([0, \infty) \times \TT^{2} ) $, the Cameron--Martin space
of the noise, the law of the solution has no obvious link
to the law of $ X $.

Second, and very relevant to our setting are a number of recent works by
Hofmanov\`a, Zhu and Zhu \cite{hofmanova2022class, hofmanova2022non, hofmanova2021globalST,
hofmanova2021global}. There, the authors establish global in time existence of
invariant solutions to equations such as the $ 3 $D Navier--Stokes equations
with space time white noise or the SQG equations within certain parameter
ranges. These results build on the equations being
super-critical and on
convex integration tools that allow to construct infinitely many
solutions at once. 

Finally, the results closest to this work both in their methods and in their
motivation concern equations linked to stochastic quantisation. For models
such as $ \Phi^{4}_{d} $ or (Euclidean) Yang--Mills, one aims to give meaning to a probability measure on a
space of distributions given by some formal expression. The idea then is to consider the
corresponding Langevin process (noisy gradient flow), which is typically a singular stochastic PDE.
If one can give a meaning to and then prove well-posedness and unique ergodicity for that SPDE, 
the desired measure can then be \textit{defined} as its (unique) invariant measure.
In the case of the $ \Phi^{4} $ models in any subcritical regime, global well-posedness 
 has been established
\cite{MourratWeber17Phi4,ChandraMoinat19,GubinelliHofmanova21Phi4,GubinelliHofmanova19Global},
making use of the strongly coercive effect of the nonlinearity. In the case of Yang--Mills 
however, while local well-posedness has recently been established
\cite{chandra2022stochastic,chandra2022langevin, chevyrev2022stochastic},
global in time well-posedness remains as challenging as it is interesting.
Morally the Yang--Mills nonlinearity should behave similarly 
to the Navier--Stokes nonlinearity, providing at least heuristically an
energy estimate and not a negative drift as is the case for the $
\Phi^{4}_{d} $ model.

In view of these considerations, establishing global well-posedness 
for equations such as Equation~\eqref{eqn-stochastic-ns-original} is
particularly interesting. Arguably, the drawback of our approach is that the
exact regularity
of $ \xi $ plays a role, and we are not able to rule out finite time explosion
if one consider an even slightly more irregular noise. On the other hand we
provide a pathwise argument for global well posedness: in particular,
no finite time explosion holds for every realisation of the noise outside a
null set, uniformly over all initial conditions and perturbations. In addition
and to he best of our knowledge for the first time,
we establish well-posedness of the equation also for $ L^{2} $ initial data. Of
course, the choice of the initial condition is intuitively a local rather then
a global question, and indeed we expect this part of our result to extend to a
broader class of equations. The link between global well-posedness and
well-posedness for critical initial data is that both rely on an energy
estimate and some kind of weak formulation of the equation.

The technique used in the present work is to introduce a dynamical 
high-low frequency decomposition, which splits the solution into an
irregular, but small, component and a more regular, but arbitrarily large
component. This
is in spirit similar to the approach taken by Gallagher and Planchon
\cite{gallagher2002global} to
establish well-posedness of the deterministic Navier--Stokes equations with
critical initial data and integrability index $ p >2 $ where the 
energy of the initial condition is allowed to be infinite. In our setting, even for
smooth initial data, the $ L^{2} $ norm
of the solution is infinite at any
positive time: this motivates our division of scales, so that our efforts
concentrate towards establishing an energy estimate for the large scale
component of the solution. In the literature on singular SPDEs, similar
decompositions have appeared in the study of $ \Phi^{4}_{d} $ models by Mourrat
and Weber \cite{MourratWeber17Phi4} and in particular also by Gubinelli and
Hofmanovà \cite{GubinelliHofmanova19Global}, where the authors introduce a
time-independent paracontrolled structure similar to ours, in order to
obtain global well-posedness in space. As a matter of fact, an argument with a somewhat similar flavour 
already appears in Nelson's original construction of the $\Phi^4_2$ measure
\cite{Nelson} (see \cite[Chapter 9]{hairer2021introduction} for a modern
account).

Ultimately, to establish the lack of finite time blow-up, we rely on a careful study of
a certain quadratic form linked to a singular operator. The latter requires a
finite, but solution-dependent, logarithmic renormalisation, leading us to the
following (very heuristic) bound:
\begin{equ}
   \partial_{t} \| u_{t} \|\lesssim \log{( \| u_{t}
   \|)} \| u_{t} \| \;,
\end{equ}
for an appropriate norm $ \| \cdot \| $. Hence we obtain a quantitative
estimate with double-exponential growth of the type 
\begin{equ}
   \| u_{t} \| \leqslant \exp (\exp (c_{t} \cdot t)) \;,
\end{equ}
where the quantity $ c_t > 0 $ depends on the noise up to time $ t $, so in
particular the growth estimate is more than double exponential.

Let us conclude with a final remark. The original local well-posedness result
by Da Prato and Debussche did not require any tools from singular SPDEs (paracontrolled 
calculus \cite{GubinelliImkellerPerkowski2015}, regularity structures
\cite{Hairer2014}, etc). However,
both our well-posedness result for critical initial conditions and our global in time
well-posedness result build on the deeper understanding of the fine structure of the
solution provided by these tools. In this instance, we will use
paracontrolled calculus for our analysis.

\subsection*{Acknowledgments}

This article was written in vast majority while TR was employed at
Imperial College London.
Financial support through the Royal Society research professorship of MH, grant number
RP\textbackslash R1\textbackslash 191065, is gratefully
acknowledged.

\subsection*{Notations}

We let $\NN = \{0, 1,2,3,\dots\}$, $ \NN_{+} = \NN \setminus \{ 0 \}$, and $ \ZZ_{*} = \ZZ \setminus \{ 0
\} $. Given a function said to depend on `space' and `time', we will always assume that the spatial variable
takes values in the $ 2 $-dimensional torus  $ \TT^2 =\RR^2 /\ZZ^2$.
Given $d \in \NN_+$ and a vector $v  \in \RR^{d}$ we write $ | v | $
for its Euclidean norm. We identify $\bM^d$, the space of $d \times d$ square
matrices with $\RR^d \otimes \RR^d$ in the usual way, and we set $u \sotimes v = \f12(u\otimes v + v\otimes u )$
Given two topological spaces $X,Y$ we write \(C(X;Y)\) for the space of
continuous functions from \(X\) to \(Y\).
For any \(k, d \in \NN\) if  \(O
\subseteq \RR^d\) we write \(C^{k}(\TT^{2} ; O)\) for the space of $k$ times
differentiable maps \(\varphi \colon \TT^{2} \to O\) (the derivatives
being continuous).
The gradient $\nabla$ and divergence $\div$ are defined as usual and,
for $\varphi \in C^1 (\TT^2 ; \bM^2)$ we set
\[ \mathrm{div} (\varphi) (x) = \Big( \sum_{i = 1}^2 \partial_i \varphi_{i, j} (x)
   \Big)_{j=1,2} \in C (\TT^2 ; \RR^2) . \]
while, for $\varphi \in C^1 (\TT^2 ; \RR^2)$ we define
$\nabla \varphi, \nablas \varphi \in C(\TT^{2}; \bM^{2})$ by
\[ (\nabla \varphi)_{i, j} = \partial_i \varphi_j\;, \qquad (\nablas
   \varphi)_{i,j} = \ft12(\partial_{i} \varphi_{j} + \partial_{j} \varphi_{i}) \;. \]
When its arguments are functions taking values in a Hilbert space, 
$\scal{\cdot,\cdot}$ denotes the corresponding $L^2$-scalar product.
Finally, the function spaces that we will need throughout the paper are 
described in Section~\ref{sec:function-spaces}. Let us merely note
that we write $ \| \varphi \| = \| \varphi \|_{L^{2}}$
for the $ L^{2} $ norm of a function $ \varphi $.

\subsubsection*{Conventions}

When the domain and target space of a function are clear from context, we will
omit them from our notations, writing for instance simply \(C^{k}\) or $L^p$.
Given a set $ \mX $ and two functions $ f, g \colon \mX \to \RR $, we write
\begin{equ}
f \lesssim g 
\end{equ}
if there exists a constant $ C> 0 $ such that $ f(x) \leqslant C g(x) $ for all
$ x \in \mX $ (similarly $ f \gtrsim g $ or $ f \simeq g$, the latter if both inequalities
hold). In order to lighten the notation and reduce the number of letters used
to denote constants, we will allow the exact value of generic
constants $
C( \vt) $ depending on some parameter $ \vt$ from a parameter set $ \Theta $,
to change from line to line.

\section{Main results}

Throughout this work the following assumptions are in force.
\begin{assumption}\label{assu:setting}
   We fix a (small enough) constant $ \kappa > 0 $.
   \begin{enumerate}
      \item \textbf{(Noise)} Let \((\Omega, \mF, \PP)\) be a probability space
supporting a space-time white noise \(\xi \colon \Omega \to (\mS^{\prime}(\RR \times \TT))^{2}\),
namely a random variable such that the \(\xi_{i}(\varphi)\) are jointly centered Gaussian for
\(i=1,2\) and \(\varphi \in \mS(\RR \times \TT^{2})\), with covariance
\[ \EE \big[ \xi_{i}(\varphi) \xi_{j}(\varphi^{\prime})  \big] = 1_{\{i=j\}}
\int_{\RR \times \TT^{2}} \varphi(t, x) \varphi^{\prime}(t, x) \ud t \ud x.\]
      
      \item \textbf{(Perturbation)} One has $ \zeta \in
\mC_{\mathrm{parab}}^{-2+3\kappa} (\RR \times \TT^2; \RR^2)$.

      \item \textbf{(Initial condition)} One has $ u_{0} \in 
         \mC^{-1 + \kappa} \cup L^{2} $ such that $ \div (u_{0}) = 0 $ and $ \Pi_{\times} u_{0} = 0 $.
   \end{enumerate}
\end{assumption}

Here the space $\mC^{-2+3\kappa}_{\mathrm{parab}}(\RR \times \TT^2; \RR^2)$ denotes the parabolically scaled 
Hölder--Besov space of space-time distributions as in \cite[Definition
3.7]{Hairer2014} with $\mf{s} = (2,1)$. This space satisfies that if $\zeta \in
\mC^{-2+3\kappa}_{\mathrm{parab}}$, then also $(t,x) \mapsto \zeta (t,x)
1_{[0,\infty)} (t) \in \mC^{-2+3\kappa}_{\mathrm{parab}}$ and the convolution
with the heat kernel $\int_0^t P_{t-s} \Pi_{\times} \zeta \ud s \in C([0,
\infty); \mC^{2 \kappa})$. We loose a $\kappa$ 
in spatial regularity (which we think of as small) in order to obtain
continuity in time.
The factor $ 2 $ in front of $ \kappa $ is simply for later convenience.
Here and in the rest of the work $\mC^\alpha$, for
$\alpha \in \RR$,
refers to the spatial Hölder--Besov space $\mC^\alpha(\TT^2;\RR^2)$ as defined
in Appendix~\ref{sec:function-spaces}.

\begin{remark}\label{rem:perturb}
   Assumption~\ref{assu:setting} allows for perturbations $ \zeta
   $ that do not lie in the Cameron--Martin space $ L^{2}_{\mathrm{loc}}([0, \infty) \times
   \TT^{2}) $ of the noise $ \xi $: in particular for such $ \zeta $, the law
   of the solution is not absolutely continuous to the solution to the 2D SNS
   equations with space-time white noise, for which global existence for
   non-anticipative initial conditions is already understood. Our assumption allows, for instance, $ \zeta $ to be a
   realisation of a noise that is white in time but smoother than $
   \xi $ in space. In addition $ \zeta $ could depend on the realisation $
   \omega \in \Omega $ of the noise $ \xi $, since our argument is completely
   pathwise.
\end{remark}

\begin{remark}\label{rem:ic}
   For $ u_{0} \in \mC^{-1 + \kappa}$ we will prove the
   existence of global mild solutions. For $ u_{0} \in L^{2} $ we will
   introduce a suitable notion of weak solution and prove global
   well-posedness for such solutions. 
\end{remark}
To simplify the study of {\eqref{eqn-stochastic-ns-original}} it is convenient
to project onto the space of divergence-free functions, thus removing the
pressure from the equation. For $k =
(k_1, k_2) \in \ZZ^2$ write $ k^{\perp} = (k_2, - k_1)$
and define the Leray projection in terms of Fourier coefficients by:
\[ \mathbf{P} \varphi (x) = \sum_{k \in \ZZ^2_{*}} e^{2 \pi \iota k \cdot
   x} \left( \hat{\varphi} (k) \cdot \frac{k^{\perp}}{| k^{\perp} |} \right)
   \frac{k^{\perp}}{| k^{\perp} |} \;, \qquad \forall \varphi \in
\mS^{\prime}(\RR \times \TT^{2}; \RR^{2}) \;, \]
where $ v \cdot w$ denotes the scalar product in $\RR^{2} $.
Applying $\mathbf{P}$ to {\eqref{eqn-stochastic-ns-original}},
we obtain
\begin{equ}[eqn-stochastic-ns-projected]
\partial_t u + \mathbf{P} (u \cdot \nabla u)  =  \Delta u +
\mathbf{P} \Pi_{\times }  (\zeta + \xi) \qquad u
(0, \cdot) = u_0 (\cdot)\;,
\end{equ}
since by assumption $\div (u_0) = 0$.
Due to the irregularity of the noise, the solution $u$ to~\eqref{eqn-stochastic-ns-projected} does not lie in $L^2$, so the
non-linearity $ \mathbf{P} (u \cdot \nabla u) = \mathbf{P} \div (u^{\otimes 2})
$ is a priori ill-defined. The key insight of
\cite{DaPrato-Debusshe-2DNS-2002} (following earlier works such as
\cite{Bourgain1996}) was to consider the solution to
{\eqref{eqn-stochastic-ns-projected}} as a perturbation of the solution $X$ to
the linear equation:
\begin{equation}\label{e:X}
   \partial_t X = \Delta X + \mathbf{P} \Pi_{\times}\xi\;, \qquad X_{0} = 0 \;.
\end{equation}
Note that contrary to a setting common in the SPDE literature we do not
choose $ X_{0} $ so that the process $ t \mapsto X_{t}  $ is stationary. 
Instead, our choice of zero initial condition will be convenient later on to deal with initial data in $
L^{2}$. 
Gaussian computations guarantee that \eqref{e:X} admits a unique solution $ X \in C ([0, \infty) ;
\mC^{- \kappa}) $ for any $ \kappa > 0 $, implying that $ X $ is barely not a
function (these calculations are by now
classical, but see also Lemma~\ref{lem:stochastic-bds} for
similar bounds).
Setting $u = X + v$, $v$ should at least formally
solve
\begin{equ}[eqn:daPrato-Debussche]
    \partial_t v  = \Delta v + \mathbf{P} \div ((v + X)^{\otimes 2}) +
    \mathbf{P} \Pi_{\times} \zeta \;,  \quad v (0, \cdot) = u_0 (\cdot)\;.
\end{equ}
Indeed the term $\div (X^{\otimes 2}) $ is defined in $C( [0, \infty) ;
\mC^{- 1-\kappa}) $ for any $ \kappa > 0 $ as a field in the second Wiener
chaos (despite the
product being a-priori ill-posed, since $ X $ has negative regularity), and parabolic
regularity estimates guarantee that one can find, at least for smooth initial
conditions, a solution $ v $ to
\eqref{eqn:daPrato-Debussche} satisfying $ v \in C ( (0, \infty); \mC^{2\kappa} )
$, for $ \kappa > 0 $ small, as is captured by the following result.

\begin{theorem}[Da Prato, Debussche
{\cite{DaPrato-Debusshe-2DNS-2002}}]\label{thm:dpd}
There exists a null set \( \mN \subseteq \Omega\) such
that for any $ \omega \not\in \mN$ and $\kappa > 0 $ the following holds. For any
$ u_{0} \in \mC^{-1 + \kappa} $
there exists a
\(T^{\mathrm{fin}}(\omega,
u_{0} ) \in (0, \infty]\) and a unique \emph{maximal} mild solution \(v(\omega)\) to
\eqref{eqn:daPrato-Debussche} on $ [0, T^{\mathrm{fin}} (\omega, u_{0})) $, with $
v (\omega, 0, \cdot) = u_{0}(\cdot)$.
\end{theorem}
The meaning of mild solutions is kept vague: we refer the reader to
Proposition~\ref{prop:wp-phi} and its (sketch of) proof. With the solution
being maximal we understand that if there
exists another \( \overline{v}(\omega)\) on an interval \([0, \overline{T}(\omega))\) that
solves \eqref{eqn:daPrato-Debussche} with the same initial condition $
u_{0}$, then 
\[ \overline{T}(\omega) \leqslant T^{\mathrm{fin}}(\omega, u_{0}) \;, \quad \overline{v} (\omega, t) =
v(\omega, t) \;, \quad \forall 0 \leqslant t < \overline{T}(\omega) \;.
\] 
We say that the maximal solution is \textit{global} for given \( \omega \) and
$ u_{0} $, if $ T^{\mathrm{fin}} (\omega, u_{0}) = \infty $.

\begin{remark}\label{rem:ic-space}
   By Sobolev embedding in dimension $ d =2 $, $ H^{\kappa} \subseteq \mC^{-1 + \kappa}
   $, so that mild solutions can deal with any initial condition with slightly better
   regularity than the critical space $ L^{2} $.
\end{remark}
Our main result concerns the existence of global solutions
for an arbitrary initial condition, almost surely.

\begin{theorem}[Global solutions]\label{thm:existence-uniqueness-global-solutions}
   There exists a null set $ \mN^{\prime} \subseteq \Omega $ such that
   $$T^{\mathrm{fin}}(u_{0}, \omega) = \infty$$ for all $\omega
      \not\in \mN^{\prime}, \kappa >0$  and  $u_{0} \in
      \mC^{- 1 + \kappa}, \zeta$ satisfying Assumption~\ref{assu:setting}.
\end{theorem}
The null set $ \mN^{\prime} $ is the one appearing in
Lemma~\ref{lem:prob-space}.
Theorem~\ref{thm:existence-uniqueness-global-solutions}
is proven at the very end of Section~\ref{sec:global-solutions}. Next we
consider initial
condition $ u_{0} \in L^{2} $. Note that in this case Theorem~\ref{thm:dpd}
does not guarantee even local in time well-posedness.

\begin{theorem}[Global high-low weak solutions]\label{thm:weak}
   For the same null set $ \mN^{\prime} \subseteq \Omega $ as in
   Theorem~\ref{thm:existence-uniqueness-global-solutions} the
   following holds. 
   For every $ \omega \not\in \mN^{\prime}, \kappa >0 $ and
   $ u_{0} \in L^{2}, \zeta $ satisfying Assumption~\ref{assu:setting} there
   exists a unique, global high-low weak
   solution to \eqref{eqn:daPrato-Debussche}, with initial condition $
   u_{0}$, in the sense of Definition~\ref{def:weak}.
\end{theorem}
This result follows from Lemma~\ref{lem:existence-weak} (existence) and
Lemma~\ref{lem:uniq-weak} (uniqueness).
The crux of the argument for both results lies in an energy estimate, based on a
dynamic high-low frequency decomposition: we will use classical energy
estimates for low frequencies and tools from singular SPDEs for high
frequencies.

\section{First steps}

To derive our energy estimate we start by iterating the Da Prato--Debussche
trick to improve as much as possible the regularity of the right-hand side.
The issue with using \eqref{eqn:daPrato-Debussche} to obtain an energy estimate
is that formally such an estimate would require us to make sense of the
pairing, which appears when differentiating in time the $ L^{2} $ norm $ t
\mapsto \| v \|_{L^{2}}^{2}$:
\begin{equ}
   \langle v,  \Delta v + \div ((v + X)^{\otimes 2}) +
   \Pi_{\times} \zeta\rangle = \langle  v,  \Delta v + \div (v^{\otimes 2} +
   X^{\otimes 2} + 2 v \sotimes X) + \mathbf{P} \Pi_{\times} \zeta \rangle \;.
\end{equ}
Since $ \zeta \in \mC^{-2 + 2 \kappa}, X \in \mC^{- \kappa} $ and, at least
locally, $ v \in \mC^{2 \kappa}$, none of the pairings 
\begin{equ}
   \langle v, \Delta v \rangle\;, \quad \langle v, \div (2 v \sotimes X)
   \rangle\;, \quad \langle v, \div (X^{\otimes 2}) \rangle \;, \quad \langle
   v, \mathbf{P} \Pi_{\times} \zeta \rangle 
\end{equ}
are well defined for generic elements of these spaces. We can 
improve the situation by introducing the solution $ Y $ to the linear equation
\begin{equ}[eqn:Y]
   \partial_{t} Y = \Delta Y + \mathbf{P} \div ( 2 X \sotimes Y + X^{\otimes 2})
   + \mathbf{P} \Pi_{\times} \zeta \;, \qquad Y(0, \cdot) = 0 \;,
\end{equ}
and then setting $ w = v - Y $ so that, setting $ D = 2 (X + Y) $, 
\begin{equ}[eqn:w]
   \partial_{t} w  = \Delta w + \mathbf{P} \div ( w^{\otimes 2} + D \sotimes w +
   Y^{\otimes 2})\;, \quad
   w(0, \cdot)  = u_{0} (\cdot) \;.
\end{equ}
Since the
worst term in \eqref{eqn:Y} is $ \mathbf{P} \Pi_{\times} \zeta $,
we have $ Y \in \mC^{2 \kappa} $, so that we expect $ w \in
\mC^{1 - \kappa} $, the worst term in \eqref{eqn:w} being given by $
\mathbf{P} \div( 2X \sotimes w) \in \mC^{-1- \kappa}$. If we now consider the
pairing
\begin{equ}[eqn:w-pairing]
   \langle w,\Delta w + \mathbf{P} \div ( w^{\otimes 2} + 2 (X + Y)\sotimes w +
   Y^{\otimes 2})\rangle \;,
\end{equ}
appearing in the time derivative of $\|w\|^2$, the only ill-defined term is 
\begin{equ}[eqn:qf]
   \langle w,\Delta w + \mathbf{P} \div ( 2 X \sotimes w ) \rangle\;.
\end{equ}
In fact, the main issue in deriving a-priori estimates on the solution is to give
meaning to this pairing. At a very heuristic (and ultimately wrong) level, we
would like to treat \eqref{eqn:qf}
as a random Dirichlet form. The problem with this approach is that the quadratic
form is not semi-bounded from below, which reflects the necessity of
renormalisation for the symmetrised version of the operator $ \varphi \mapsto
\Delta \varphi + 2 \div (X \sotimes \varphi) $. This problem will
be addressed by the already mentioned division of scales, so that at a fixed time we
will require only a finite (but solution-dependent!) logarithmic
renormalisation constant.

The second issue to address is how much regularity
can be found in this quadratic form: by comparison, the quadratic form
associated to the Laplacian guarantees one degree of regularity since $ \langle
f, \Delta f \rangle = - \| \nabla f \|^{2} $.
In our case, for $ \mu  \gg 1 $, the
following resolvent (here $ \mathrm{sym} $ stands for the symmetric part of the
operator) is expected to be a bounded operator
\begin{equ}
   ( \Delta  + \mathbf{P} \div (2 X \sotimes \cdot) - \mu
   )^{-1}_{\mathrm{sym}} \colon L^{2} \to H^{1 - \kappa}\;.
\end{equ}
In particular, because of the stochastic terms, it is expected to be less regularising than the
Laplacian alone (which is bounded into $ H^{2} $). Therefore, we could
expect the quadratic form above to be bounded from above as follows, for some
(random) $ c > 0 $
\begin{equ}
   \langle w,  \Delta w + \mathbf{P} \div (2 X  \sotimes w) \rangle
   \lesssim - \| w \|_{H^{1/2 - \kappa/2}}^{2} + c \mu \| w \|^{2} \;.
\end{equ}
This is a significant loss of regularity compared to the
Laplacian and such a bound would not be sufficient to deduce our result. To
solve this issue we observe that our argument only requires a fraction of
regularity to treat the singular term $ \div (X \sotimes w) $, so we split
the quadratic form into
\begin{equ}
   \langle w , \Delta w + \mathbf{P} \div (2 w \sotimes X) \rangle =
   \ft12 \langle w, \Delta w  \rangle + \langle w, \ft12 \Delta
   w + \mathbf{P} \div( 2 w \sotimes X) \rangle \;.
\end{equ}
The first term yields an $ H^{1} $ bound and the
second term is controlled by the division of scales. Note that this division is
extremely artificial and highlights that our argument is somewhat rough and
does not optimally capture the actual small scale structure of the solution.

\subsection{Intermezzo: collecting the stochastic terms}

In order to reduce the number of norms that we will later use in our bounds, it
will be convenient to collect all stochastic quantities as elements of a large
Banach space and use only one norm on that space. So far we have considered the
following time-dependent processes, with associated ``magnitude'' $
\mathbf{L}_{t} $:
\begin{equ}
   t \mapsto (X_{t} , Y_{t}) \in \mC^{- \kappa} \times
   \mC^{2 \kappa} \;, \qquad \mathbf{L}_{t}^{\kappa} = 1+ \sup_{0 \leqslant s \leqslant t}
   \left\{ \| X_{s} \|_{\mC^{- \kappa}}+ \| Y_{s} \|_{\mC^{2 \kappa}} \right\} \;.
\end{equ}
for $ t \in [0, \infty), \kappa > 0 $.
In addition, in Section~\ref{sec:sym-op} we will consider the time-dependent Anderson-type
operator $ \frac{1}{2} \Delta + 2 \nabla_{\mathrm{sym}}X_{t} $. We therefore additionally
consider the following ``enhanced noise'' process, for a given parameter $ \lambda \geqslant 1 $
and $ t \in [0, \infty) $:
\begin{equ}[e:defProc]
    t \mapsto
   (2 \nabla_{\mathrm{sym}} \mL_{\lambda} X_{t}, (2 \nabla_{\mathrm{sym}}
   \mL_{\lambda} X_{t}) \reso P^{\lambda}_{t} - \mf{r}_{\lambda} \mathrm{Id} )
   \;,
\end{equ}
taking values in $ \mC^{-1 - \kappa} \times \mC^{- \kappa} $ (see Definition~\ref{def:cut-off}
for the definition of the projection $\mL_\lambda$, \eqref{eqn:ren-const} for the definition of
$\mf{r}_{\lambda}$, Lemma~\ref{lem:stochastic-bds} for the definition of $P^{\lambda}_{t}$,
and 
   Section~\ref{sec:para} for the definition of the ``resonant product''
$\reso$). Then
we measure the magnitude of the enhanced noise together with the processes $
X $ and $ Y $ via
\begin{equ}[e:def-N]
   \mathbf{N}_{t}^{\kappa} = \mathbf{L}_{t}^{\kappa} + \sup_{0 \leqslant s
   \leqslant t}   \sup_{i \in \NN} \Big\{\| (2 \nabla_{\mathrm{sym}
      } \mL_{\lambda^{i}}X_{s} ) \reso P^{\lambda^{i}}_{s} -
   \mf{r}_{\lambda^{i}} \mathrm{Id} \|_{\mC^{- \kappa}} \Big\} \;,
\end{equ}
where $ \{ \lambda^{i} \}_{i \in \NN} $ is defined in \eqref{e:def-lambda}. 

\begin{lemma}\label{lem:prob-space}
   Let $ (\Omega, \mF, \PP) $ be the probability space as in
   Assumption~\ref{assu:setting}. There exist a null set $ \mN^{\prime} \subseteq
   \Omega $ such that 
   \begin{equ}
      \mathbf{N}_{t}^{\kappa}  (\omega) < \infty\;, \qquad \forall \omega \not\in
      \mN^{\prime}  \;, t \geqslant 0 \;, \kappa > 0\;.
   \end{equ}
\end{lemma}

\begin{proof}
   For the terms involving $ \nabla_{\mathrm{sym}} X $ see
   Lemma~\ref{lem:stochastic-bds}. The bounds on all other terms follow along
   similar lines.
\end{proof}
For clarity, we collect all the stochastic terms required in our analysis in
the following table.

\begin{center}
\begin{tabular}{lll}
\toprule
\thead{Process} & \thead{Definition}  &\thead{Regularity} \\
\midrule
$ X $ & $ (\partial_{t} - \Delta) X = \mathbf{P} \Pi_{\times} \xi $ & $
\mC^{- \kappa}$ \\
$ Y $ & $ (\partial_{t} - \Delta) Y = \mathbf{P} \div (2 X \sotimes Y +
X^{\otimes 2}) + \mathbf{P} \Pi_{\times} \zeta$ & $ \mC^{2 \kappa} $ \\
$ Q $ & $ (\partial_{t} - \Delta) Q = 2 X $ & $ \mC^{2 - \kappa}$ \\
$ P $ & $ (- \frac{1}{2} \Delta +1 ) P = 2 \nabla_{\mathrm{sym}} X $ & $
\mC^{1 - \kappa}  $ \\
\bottomrule
\end{tabular}
\end{center}

\subsection{Recap: local well-posedness}
Before we move on, let us recall the local well-posedness result for
\eqref{eqn:w}. The proof of a very similar result can be found in
\cite{DaPrato-Debusshe-2DNS-2002} and is by now classical. For $ \gamma, T > 0 $ and $
\beta \in \RR $, we consider the Banach space
\begin{equ}
   \mM^{\gamma}_{T} \mC^{\beta} \subseteq C([0, T]; \mS^{\prime}
   (\TT^{2})) \;, \qquad \| f \|_{\mM^{\gamma}_{T} \mC^{\beta}} = \sup_{0 \leqslant t
   \leqslant T} t^{\gamma} \| f_{t} \|_{\mC^{\beta}} \;.
\end{equ}
We then say that $ w \in \mM^{\gamma} \mC^{\beta} $ is a mild solution to
\eqref{eqn:w} if
   \begin{equ}
      w_{t} = P_{t} w_{0} + \int_{0}^{t} P_{t -s}
      \mathbf{P} \div (w^{\otimes 2} + D \sotimes w + Y^{\otimes 2}
      ) \ud s \;,
   \end{equ}
   where the definition of the products $
   w^{\otimes2} $ and $ D \sotimes w $ has to be justified, depending on the
   choice of the parameters $ \beta$ and $\gamma $.

\begin{proposition}\label{prop:wp-phi}
Fix any $ 0 < \kappa < 1/2$, set $\gamma = 1 - \kappa/2 $, and assume that $ D \in C
([0, \infty) ; \mC^{- \kappa}) $ and $ Y^{\otimes 2} \in C ([0, \infty); \mC^{
2 \kappa}) $. Then for all $ u_{0} \in
   \mC^{-1 + 2\kappa} $,
   \eqref{eqn:w} admits a unique mild solution in the space $
   \mM^{\gamma/2}_{T^{\mathrm{fin}}} \mC^{3 \kappa/2}$, up to a maximal time $
   0 < T^{\mathrm{fin}}(\mathbf{L}_{t}^{\kappa}, u_{0})  \leqslant \infty$.
\end{proposition}

\begin{remark}\label{rem:final-phi}
By Lemma~\ref{lem:prob-space}, there exists a nullset \( \mN \subseteq
\mN^{\prime} \) such that $
\mathbf{L}^{\kappa}_{t}(\omega) < \infty  $ for all $ t \geqslant 0 $ and $
\omega \not\in  \mN $, so that the proposition above applies to our setting on $
\mN $.
   Moreover, the maximal local existence time $ T^{\mathrm{fin}} $ is the same
   as in Theorem~\ref{thm:dpd}, since mild solutions to
   \eqref{eqn:w} are
   equivalent to mild solutions to \eqref{eqn:daPrato-Debussche} through the
   mapping $ w \mapsto w + Y $.
\end{remark}

\section{A high-low energy estimate}

We now analyse the most problematic term in deriving an energy
estimate for $ w $, namely the quadratic form \eqref{eqn:qf}.
Since $ D \in \mC^{- \kappa} $,  we expect
 the solution $ w $ to be of regularity no better than $ \mC^{1 - \kappa} $, also for smooth initial
data. Hence to make sense of \eqref{eqn:qf} there is no chance in treating the two terms
$ \langle w, \Delta w \rangle $ and $ \langle w,
\mathbf{P} \div ( D \sotimes w) \rangle $ separately since both terms would be infinite. Instead, we have to exploit that there are cancellations
 between these two terms which make the quadratic form finite.
Before we continue, let us assume that  $
u_{0} \in L^{2} $: in any case Proposition~\ref{prop:wp-phi} guarantees
that $ w_{t} \in L^{2} $ for any $ t >0 $ up to the blow-up time $
T^{\mathrm{fin}} $, also if $ u_{0} $ has worse regularity.
\begin{assumption}\label{assu:L2IC}
   Throughout this section we work under the assumption that $ u_{0} \in
   L^{2} \cap \mC^{-1 + \kappa} $ for some $ \kappa >0 $.
\end{assumption}

\subsection{High frequency paracontrolled decomposition}\label{sec:hl-freq}
One way to observe the above mentioned cancellation is to look deeper into the structure of the
solution $ w $, using paracontrolled calculus to obtain a nonlinear expansion in 
terms of $ D $. Let us define \(w^{\sharp} \) by
\begin{equ}
  (\partial_{t} - \Delta) Q  = 2 X \;, \qquad Q_{0} = 0 \;, \qquad\label{e:Q}
	w  = \mathbf{P} \div( w \para Q) + w^{\sharp} \;,  
\end{equ}
where the paraproduct $ \para $ is defined in Section~\ref{sec:para}, and 
we note that $ Q \in \mC^{2 - \kappa} $.
Then $ w^{\sharp} $ solves
\begin{equ}[eqn:phi-sharp]
   \partial_{t} w^{\sharp} = \Delta w^{\sharp} + \mathbf{P} \div
   ( w^{\otimes 2} + D \sotimes w - 2 X \rpara w +
   C^{\para }(w, Q) +  Y^{\otimes 2}) \;,
\end{equ}
with the commutator $C^{\para }$ defined by
\begin{equs}[e:com]
   C^{\para}(f, g)& = (\partial_{t} - \Delta) (f \para g) - f \para
   (\partial_{t} - \Delta) g \\
   & = ( (\partial_{t}- \Delta)f ) \para g + \mathrm{Tr} [ (\nabla f)
   \para (\nabla g) ]\;.
\end{equs}
The term $ C^{\para}(w, Q) $ is expected to lie in \(\mathcal{C}^{1-
2\kappa}\) (see \cite[Lemma 2.8]{GubinelliPerkowski2017KPZ}, although here we are not using
the parabolically scaled paraproduct, so the estimate will follow along a different
line). Therefore, collecting all
regularities we expect that $ w^{\sharp} \in \mC^{1+ 2 \kappa} $,
since the worst regularity term in the divergence is given by $ Y \sotimes
w \in \mC^{2 \kappa} $, assuming $ \kappa $ sufficiently small (recall that $D = 2(X+Y)$). This
means that we have singled out the most irregular part of the solution $
w$ and we can now attempt to write an energy estimate for $
w^{\sharp} $. A na\"ive attempt will fail though, because now the pairing
\begin{equ}
   \langle w^{\sharp} ,  \Delta w^{\sharp} + \mathbf{P} \div
   ( w^{\otimes 2} + D \sotimes w - 2 X \rpara w +
   C^{\para }(w, Q) +  Y^{\otimes 2})\rangle 
\end{equ}
is cubic in the norm of $ w $, since the nonlinear term
\begin{equ}
   \langle w^{\sharp} , \mathbf{P} \div ( w^{\otimes 2}) \rangle
   =\langle w^{\sharp} , \mathbf{P} \div ( (w - w^{\sharp}
   )^{\otimes 2} + 2 w \sotimes (w - w^{\sharp})) \rangle 
\end{equ}
does not cancel out completely. On the other hand, if we knew that the
irregular part $ w - w^{\sharp} $ is of order one in some
appropriate norm, then we would be able to obtain an estimate that is quadratic
in the norm of $ w $, or in this case equivalently the norm of $ w^{\sharp} $.
This is our aim and the approach we will follow to “make” the
irregular part small is to take into account the paracontrolled structure only
in high frequencies, where “high” will be defined in terms of the $ L^{2} $ norm of $ w $.

\subsubsection{High and low frequency projections}
We start by introducing high and low frequency projections, together with
some simple estimates.

\begin{definition}\label{def:cut-off}
For any \( \lambda> 0 \), define the projections
\begin{equ}
  \mH_{\lambda}\colon \mS^{\prime}(\TT^{2}; \RR^{2}) \to \mS^{\prime} (\TT^{2};
\RR^{2}) \qquad    \mathcal{L}_{\lambda}  \colon \mS^{\prime}(\TT^{2}; \RR^{2}) \to
\mS(\TT^{2}; \RR^{2})
\end{equ}
by respectively
$\mH_{\lambda} w = \check{\mf{h}}_{\lambda} * w$, and $\mL_{\lambda}
w = w - \mH_{\lambda} w = \check{\mf{l}}_{\lambda} * w,$
where \( \check{\mf{h}}_{\lambda}, \check{ \mf{l}}_{\lambda} \in \mS^{\prime} (\TT^{2};
\RR^{2}) \) are defined as the
Fourier inverses 
\[ 
   \check{\mf{h}}_{\lambda}(x) = \mF^{-1} \big( \mf{h}(| \cdot |/ \lambda)
   \big)(x)\;, \qquad \check{\mf{l}}_{\lambda} (x) =  \mF^{-1} \big( \mf{l}(| \cdot |/ \lambda)
   \big)(x) \;.
\] 
for smooth functions \( \mf{h}, \mf{l} \colon [0, \infty) \to [0,
\infty) \) satisfying
\begin{equ}
\mf{h}(r) = 1 \;, \ \text{ if } r \geqslant 1 \;, \qquad \mf{h}(r)
= 0\;, \ \text{ if } r \leqslant \frac{1}{2} \;, \qquad \mf{l} = 1 -
\mf{h} \;.
\end{equ}
\end{definition}
The next result states that
we can gain regularity in low frequencies by paying a price in powers of $
\lambda $. The spaces $ \mC^{\alpha}_{p} $, for $ p \in [1, \infty] $ and $
\alpha \in \RR $ are the Besov spaces with integrability parameter $ p
$ introduced in Appendix~\ref{sec:function-spaces}.
\begin{lemma}\label{lem:low-freq}
  For any $p \in [1, \infty]$ and $\beta > \alpha$ one can estimate
  uniformly
  over $\lambda \geqslant 1$:
  \[ \| \mL_{\lambda} \varphi \|_{\mC^{\beta}_p} \lesssim \lambda^{
   \beta - \alpha} \| \varphi \|_{\mC^{\alpha}_p}\;, \qquad \forall \varphi \in
  \mC^{\alpha}_{p} \;. \]
\end{lemma}

\begin{proof}
   We can write estimate, for some $ c > 0 $ by Young's inequality for
   convolutions applied to $ \mL_{\lambda} \Delta_{j} \varphi =
   \check{\mf{l}}_{\lambda} * \Delta_{j} \varphi $
  \begin{equ}
     \| \mL_{\lambda} \varphi \|_{\mC^{\beta}_p}  =  \sup_{j \leqslant  
    \log_2 (\lambda) + c } 2^{j \beta} \| \Delta_j \varphi \|_{L^p} \lesssim 
    \lambda^{\beta - \alpha} \sup_{j \leqslant \log_2 (\lambda) + c } 2^{j
    \alpha} \| \Delta_j \varphi \|_{L^p} \;,
  \end{equ}
  from which the result immediately follows.
\end{proof}
Similarly we can gain powers of $ \lambda $ in high
frequencies by paying a price in regularity.
\begin{lemma}\label{lem:high-freq-reg-loss}
   For any $ p \in [1, \infty] $ and $\beta > \alpha$ one can estimate
   uniformly  $\lambda \geqslant 1$:
   \[ \| \mH_{\lambda} \varphi \|_{\mathcal{C}^{\alpha}_p} \lesssim
   \lambda^{\alpha- \beta} \| \varphi \|_{\mathcal{C}^{\beta}_p} \;, \qquad
\forall \varphi \in \mC^{\beta}_{p} \;. \]
\end{lemma}

\begin{proof}
   As above, we can bound for some $ c > 0 $ 
  \begin{equ}
     \| \mH_{\lambda} \varphi \|_{\mathcal{C}^{\alpha}_p}  = 
     \sup_{j \geqslant - 1} 2^{j \alpha} \| \Delta_j \mH_{\lambda}
    \varphi \|_{L^p} =  \sup_{j \geqslant \log_2 (\lambda) - c} 2^{j \alpha} \|
    \Delta_j \mH_{\lambda} \varphi \|_{L^p} .
  \end{equ}
  Now we can simply estimate
  \begin{equs}
     \| \Delta_j \mH_{\lambda}  \varphi \|_{L^p}  =  \| \mH_{\lambda} \Delta_j
     \varphi\|_{L^p}
     & \leqslant  \| \Delta_j \varphi \|_{L^p} + \| \lambda^{d}
     (\mathcal{F}^{-1} \mf{h}) (\lambda \cdot) \ast \Delta_j \varphi 
    \|_{L^p} \\
    & \leqslant  2 \| \Delta_j \varphi \|_{L^p} \;.
  \end{equs}
  Plugging this into the previous bound we obtain:
  \begin{equ}
     \| \mH_{\lambda} \varphi \|_{\mathcal{C}^{\alpha}_p}  = 
     \sup_{j \geqslant  \log_2 (\lambda)-c} 2^{j \alpha} \| \Delta_j
     \mH_{\lambda} \varphi \|_{L^p} \lesssim  \lambda^{\alpha- \beta} \|
     \varphi \|_{\mathcal{C}^{\beta}_p} \;,
  \end{equ}
  as required.
\end{proof}

\subsubsection{Construction of the high-frequency paracontrolled
decomposition}
For a time-dependent frequency
level $ \lambda_{t} \geqslant 1 $ that
will be introduced later on and with $Q$ as in \eqref{e:Q}, 
let us define the high-frequency control \( Q^{\mH} \) and the high-frequency
component $ w^{\mH} $ of $ w $ by
\begin{equ}[eqn:def-phi-high]
   Q^{\mH}_{t} = \mH_{\lambda_{t} }Q_{t} \;, \qquad
w^{\mH} = \mathbf{P} \div( w \para Q^{\mH}) \;, \qquad
w^{\mL} = w - w^{\mH} \;.
\end{equ}
Clearly, $ w^{\mL} $ should be interpreted as the low frequency component of $
w $.
In view of the two lemmas above, the gain in this decomposition is that if the frequency $
\lambda_{t} $ is large, then the control $ Q^{\mH} $ is relatively small,
provided that we are willing to measure it with a worse regularity. 
In particular, we can bound by
Lemma~\ref{lem:high-freq-reg-loss}:
\begin{equ}[e:QH-bd]
   \| Q_{t}^{\mH} \|_{\mC^{2 - \kappa - \delta}} \lesssim \lambda_{t}^{-
   \delta} \| Q_{t} \|_{\mC^{2- \kappa}} \lesssim \lambda_{t}^{-  \delta} \|
   X_{t} \|_{\mC^{- \kappa}} \;.
\end{equ}
To make sure that $ w^{\mH} $ is of order one, independently of the size
of $ w $, it would be convenient to choose $ \lambda_{t} \simeq (1 + \|
w_{t} \|) $. Of course, such a decomposition eventually shifts
the problem to the analysis of the low-frequency component $ w^{\mL} $.
For this purpose, such a choice of $ t \mapsto \lambda_{t}
$ is not entirely convenient, because in deriving an equation for $
w^{\mL} $ we will end up differentiating $ \lambda_{t} $ in time, which
leads to a tedious term involving quantities such as the quadratic form
\eqref{eqn:qf}, which is precisely what we set out to avoid. Instead we consider a
discretised version of $ t \mapsto (1 + \| w_{t} \|)^{\mf{a}} $, for a
suitable $ \mf{a} > 0 $. 

\begin{definition}\label{def:stop-time-and-lambda}
   Fix a parameter $ \mf{a}> 0 $ and consider any initial condition $
   u_{0} $ satisfying Assumption~\ref{assu:L2IC}. Let us introduce the sequence of stopping times $ \{ T_{i} \}_{i \in \NN} $, with
\begin{equ}
   0 \eqdef T_{0} \leqslant  T_{1} \leqslant  \dots \leqslant  T_{i} \leqslant  \dots \;,
\end{equ}
defined for any $ \omega \in \Omega $ and $ u_{0} \in L^{2} \cap \mC^{-1 +
\kappa} $ as follows. 
For $ i \in \NN \setminus \{ 0 \} $ define
\begin{equ}[eqn:def-T]
   T_{i+1} (\omega, u_{0}) = \inf \{ t \geqslant T_{i}  \; \colon \; \| w_{t} \|
   \geqslant i+1 \} \wedge T^{\mathrm{fin}} (\omega, u_{0}) \;,
\end{equ}
with $ w $ solving \eqref{eqn:w} and $ T^{\mathrm{fin}} (\omega, u_{0}) $ as in
Theorem~\ref{thm:dpd}.
Then if we set
\begin{equ}[e:i0]
   i_{0}(u_{0} ) = \max \{ i \in \NN  \; \colon \; i \leqslant \| u_{0} \| \}
   \;, 
\end{equ}
it holds that
$T_{i} = 0$ if and only if $i \leqslant i_{0} (u_{0})$.
Finally, for any $ i \in \NN$ set
\begin{equ}[e:def-lambda]
   \lambda^{i} \eqdef (1 + i)^{\mf{a}}\;, \qquad \lambda_{t}
   \eqdef \begin{cases}
      (1 + \lceil \| u_{0} \| \rceil)^{\mf{a}} \;, & \text{ if } t = 0  \;,\\
      (1 + \| w_{T_{i}} \|)^{\mf{a}}\;, & \text{ else, for all } \, T_{i} \leqslant
      t < T_{i+1} \;.
   \end{cases}
\end{equ}
\end{definition}
Since $ u_{0}\in L^{2} $ by Assumption~\ref{assu:L2IC}, we have $
i_{0}(u_{0}) < \infty$. Moreover $ \lambda_{t} $ is defined so that 
$ \lambda_{t} = \lambda^{i} $ for all $ T_{i} \leqslant t <
T_{i+1} $ and $ i \geqslant i_{0}(u_{0}) $.

\begin{remark}\label{rem:a}
We have defined the discretised frequency level $ t \mapsto \lambda_{t} $ in
such a way that $ \lambda_{t} \in \{ \lambda^{i} \}_{i \in \NN} $. In
particular, it belongs to a fixed countable set \emph{independent of initial
conditions}, which will be of use in the approximation of the singular operator
in Section~\ref{sec:sym-op} (else the null set in
Lemma~\ref{lem:stochastic-bds} could depend on the initial condition $
u_{0} $).
   Moreover, we introduced the parameter $ \mf{a} $, because it turns out that $
   \mf{a} = 1 $ (arguably the most natural choice) is not enough for our
   purposes. Instead choosing any $ \mf{a} \in (2, 3] $ is sufficient. We left
   $ \mf{a}$ as a free parameter, so that the reader can follow at what point the
   condition $ \mf{a} > 2 $ is required.
\end{remark}
Next we make use of the structure we have introduced so far to control the
high-frequency term $ w^{\mH} $.

\begin{lemma}\label{lem:high-feq-ref}
   For any $ \delta > 0$ and $\omega \not\in \mN^{\prime} $ there exists a $ C(\delta) > 0 $ such that
   \begin{equ}
      \| w_{t}^{\mH}(\omega) \|_{H^{1 - 2 \kappa-  \delta}} \leqslant
      C(\delta) (1+ \| w_{t}(\omega)\|)^{1 - \mf{a} \delta}
      \mathbf{N}_{t}^{\kappa}(\omega) \;, \qquad \forall 0
     \leqslant t < T^{\mathrm{fin}}(\omega, u_{0}) \;.
   \end{equ}

\end{lemma}
Observe that the formulation of the lemma above allows for $ t $ to depend on $
\omega $, and does not require it to be a stopping time. 

\begin{proof}
   The estimate follows from Lemma~\ref{lem:parap},
   Lemma~\ref{lem:high-freq-reg-loss}, the
   definition of $ w^{\mH} $ in \eqref{eqn:def-phi-high} and $
   \lambda_{t} $ in Definition~\ref{def:stop-time-and-lambda}, since
\begin{equs}
   \| w^{\mH}_{t} \|_{\mC^{1 - \kappa - \delta}_{2}} & \lesssim \| w_{t}\para
   Q^{\mH}_{t} \|_{\mC^{2 - \kappa - \delta}_{2}} \lesssim \|
   w_{t} \| \| Q^{\mH}_{t} \|_{\mC^{2 - \kappa - \delta}} \\
   & \lesssim \| w_{t} \| \lambda^{-
   \delta}_{t} \| Q_{t} \|_{\mC^{2 - \kappa }}
   \lesssim (1+ \| w_{t} \|)^{1 - \mf{a} \delta} \| X_{t} \|_{\mC^{- \kappa}} \;.
\end{equs}
Here, in the first estimate of the second line we made use of the continuous
embedding $ \mC^{\alpha}_{2} \subseteq H^{\alpha - \kappa}_{2} $ for any $
\kappa > 0 $.
\end{proof} 

\subsection{Structure of the low-frequency energy estimate}

In a nutshell, the question of global existence amounts then to proving that $ \lim_{i \to
\infty} T_{i} = \infty$ (of course we still have to prove that $
T^{\mathrm{fin}}  $ coincides with the blow-up time of the $ L^{2} $ norm). To
obtain such a result we now establish an $
L^{2} $ energy estimate on the low-frequency component $ w^{\mL} $.
Let us fix an $ i \in \NN, i > i_{0}$ and consider $ t \in [T_{i}, T_{i+1}) $, whenever
$ T_{i}< T_{i +1} $, with the stopping
times $ T_{i} $ defined as in \eqref{eqn:def-T}. We find, in analogy to
\eqref{eqn:phi-sharp}
\begin{equs}[e:wL]
   \partial_{t} w^{\mL} = \Delta w^{\mL} &+ \mathbf{P} \div
   (w^{\otimes 2} + D \sotimes w - 2 ( \mH_{\lambda_{t}} X ) \rpara w)\\
   & + \mathbf{P} \div( C^{\para }(w, Q^{\mH}) +  Y^{\otimes 2}) \;,
\end{equs}
with initial condition $ w^{\mL}_{0} = w_{0} - \mathbf{P} \div ( w_{0} \para
Q^{\mH}_{0}) = u_{0} $, so that since $D = 2(X+Y)$
\minilab{e:fourTerms}
\begin{equs}
   \partial_{t} \| w^{\mL} \|^{2} &=  2\langle w^{\mL}, \Delta
   w^{\mL} +  \div ( 2 ( \mL_{\lambda_{t}} X ) \sotimes
   w^{\mL}) \rangle \label{eqn:bd-operator} \\ 
   &\qquad +  2\langle w^{\mL},  \div ( 2
   (\mH_{\lambda_{t}} X) \sotimes w^{\mL} - 2
   (\mH_{\lambda_{t}} X ) \rpara w^{\mL} ) \rangle   \label{eqn:bd-interpolation} \\
   &\qquad + 2 \langle w^{\mL},  \div ( 2 X \sotimes w^{\mH} - 2
   (\mH_{\lambda_{t}} X ) \rpara w^{\mH} )\rangle\label{eqn:bd-3} \\ 
   &\qquad + 2 \langle w^{\mL},   \div
   (w^{\otimes 2} + 2 Y \sotimes w  +
   C^{\para }(w, Q^{\mH}) +  Y^{\otimes 2})\rangle \;, \label{eqn:bd-4}
\end{equs}
with $ C^{\para} $ as in \eqref{e:com}.
We will treat the four terms in \eqref{e:fourTerms} separately.
The term (\ref{eqn:bd-operator}) gives rise to a paracontrolled
quadratic form, which will need logarithmic (in $ \lambda_{t} $)
renormalisation. To bound the cubic term in
\eqref{eqn:bd-4} we use our decomposition, in combination with
Lemma~\ref{lem:high-feq-ref}.

Let us start with the quadratic form in \eqref{eqn:bd-operator}. Since
both $ X $ and $ w^{\mL} $ are divergence-free, we have
\begin{equs}
   \langle w^{\mL}, \tf{1}{2}  \Delta w^{\mL} + \div ( 2 (
   \mL_{\lambda_{t}} X_{t} ) \sotimes w^{\mL}) \rangle & = \langle w^{\mL},
   \tf{1}{2}\Delta w^{\mL} + [ 2 \nabla (\mL_{\lambda_{t}} X_{t}
   ) ] w^{\mL}\rangle \\
   & = \langle w^{\mL}, \tf{1}{2} \Delta
   w^{\mL} + [ 2 \nabla_{\mathrm{sym}} (\mL_{\lambda_{t}} X_{t}) ]
   w^{\mL}\rangle\;.
\end{equs}
The factor $ 1/2 $ in front of the Laplacian is not a typo, we will split the
Laplacian into two terms, see \eqref{e:lap-split} below. 
This leads us to consider the following time-dependent family of operators:
\begin{equ}
   \mA_{t} \eqdef \frac{1}{2} \Delta +  2 \nabla_{\mathrm{sym}} X_{t} -
   \infty\;,
   \qquad \forall t \geqslant 0 \;,
\end{equ}
where the ``$ \infty $'' indicates the necessity of renormalisation. More precisely, 
$ \mA_{t} $ will be constructed as the limit as $ \lambda \to\infty$ of the operators:
\begin{equ}[eqn:def-A]
   \mA_{t}^{\lambda} \eqdef \frac{1}{2} \Delta +
   2\nabla_{\mathrm{sym}} \mL_{\lambda} X_{t} - \mf{r}_{\lambda} (t)
   \mathrm{Id} \;, \qquad \forall t \geqslant 0 \;.
\end{equ}
Here $ \mf{r}_{\lambda}(t) $ is the renormalisation constant defined in
\eqref{eqn:ren-const} below (the fact that such a limit exists is the content of
Proposition~\ref{prop:def-A} and Lemma~\ref{lem:stochastic-bds}).
For clarity let us formally write the action of these operators in
components:
\begin{equ}[eqn:def-A-l]
   (\mA_{t}^\lambda w)_{i}  =  \sum_{j = 1}^{2} \delta_{i, j} \frac{1}{2}\Delta w_{j} +
   [2 (\nabla_{\mathrm{sym}} \mL_{\lambda} X_{t})_{i, j} - \mf{r}_{\lambda}(t) \delta_{i, j} ] w_{j} \;, \quad \forall w
   \in D( \mA_{t}) \;.
\end{equ}
This allows us to rewrite \eqref{eqn:bd-operator} as
\begin{equs}[e:lap-split]
   \langle w^{\mL}, \Delta w^{\mL} + & \div ( 2 (
   \mL_{\lambda_{t}} X_{t} ) \sotimes w^{\mL}) \rangle \\
& = - \frac{1}{2}
   \| w^{\mL} \|_{H^{1}}^{2} +  \langle
   w^{\mL}, \mA_{t}^{\lambda_{t}} w^{\mL}\rangle +
   \mf{r}_{\lambda_{t}} \| w^{\mL} \|^{2} \;.
\end{equs}
The remaining terms in \eqref{e:fourTerms} will be treated as perturbations of this term. 
At this point we can thus already provide the heuristics of our approach, assuming
\eqref{eqn:bd-interpolation}--\eqref{eqn:bd-4} vanish.
In this case we are left with the following bound, from the
definition of $ \mf{r}_{\lambda_{t}} (t) $ in
\eqref{eqn:ren-const} (where the constant \( c >0 \) appears) for $ t \in
[T_{i}, T_{i+1}) $
\begin{equs}
   \partial_{t} \| w^{\mL} \|^{2} & = - \| w^{\mL} \|_{H^{1}}^{2} + 2 \langle w^{\mL},
   \mA_{t}^{\lambda_{t}} w^{\mL} \rangle + 2 \mf{r}_{\lambda_{t}} (t) \|
   w^{\mL} \|^{2} \\
   & \leqslant  - \| w^{\mL} \|_{H^{1}}^{2} + 2 \langle w^{\mL},
   \mA_{t}^{\lambda_{t}} w^{\mL} \rangle + \mf{a}c \log{(1 + \| w_{T_{i}}
   \|^{2})} \|
   w^{\mL} \|^{2} \;.
\end{equs}
In addition, by Proposition~\ref{prop:def-A}, there exists a continuous map $
\mathbf{m} \colon \RR_{+} \to \RR_{+} $ such that the operator $ \mA_{t} -
\mathbf{m}_{t} $ is negative. By Lemma~\ref{lem:high-feq-ref}
with $ \delta = 1/ \mf{a} $ we additionally have $ \| w^{\mH} \| \simeq 1  $ so that $ \| w \| \simeq \|
w^{\mL} \| + 1$, and we conclude
\begin{equ}
   \partial_{t} \| w^{\mL} \|^{2}  \leqslant 2 \mathbf{m}_{t}
   \| w^{\mL} \|^{2} +  c \log{(1 + \| w^{\mL}_{T_{i}}
   \|^{2})} \| w^{\mL} \|^{2} \;.
\end{equ}
Roughly, this calculation shows that the norm grows at most like the solution
to the ODE
\begin{equ}
   \dot{z}_{t} = c \log{ (z_{t})} z_{t} \;,
\end{equ}
for some $ c >0 $, which has double-exponential growth but does not blow up in finite time.
More rigorously, we obtain that for $ t \in [T_{i}, T_{i+1}) $ 
\begin{equ}
   \| w^{\mL}_{t} \|^{2} \leqslant \| w_{T_{i}}^{\mL}
   \|^{2} \exp \left\{ \int_{T_{i}}^{t} 2 \mathbf{m}_{s} +  c \log{(1 + \|
   w^{\mL}_{T_{i}} \|^{2})} \ud s \right\} \;.
\end{equ}
Now for $ t \geqslant 0 $ write $
\overline{\mathbf{m}}_{t} = \max \{ \mathbf{m}_{s}  \; \colon \; s \in
[0, t] \} $. If for the sake of our argument we assume that the blow-up
time $ T^{\mathrm{fin}} $ coincides with the blow-up time of the $
L^{2} $ norm of $ w $ (we will prove this in Corollary~\ref{cor:explosion-condition}),
then our aim is to prove that $ T_{i} \uparrow \infty $. In particular, if $
T^{\mathrm{fin}} < \infty $, we would have $
T^{i} < T^{\mathrm{fin}} < \infty $ for all $ i \in \NN $. On the other hand,
we can bound for any $ i \geqslant 1 $, by using that $
\| w^{\mL}_{T_{i}} \|^{2} \simeq i^{2}$ and $ \Delta_{i} = T_{i+1} - T_{i}$ 
\begin{equ}
   (i +1)^{2} \leqslant i^{2} \exp \left\{ \Delta_{i} \left[
      2 \overline{\mathbf{m}}_{T^{\mathrm{fin}}} +  c \log{(2
   i^{2})} \right] \right\} \;,
\end{equ}
whence
\begin{equ}
\Delta_{i} \geqslant  \frac{\log{ \left( \frac{(i + 1)^{2}}{i^{2}} \right)}
   }{ 2 \overline{\mathbf{m}}_{T^{\mathrm{fin}}} +  c \log{(2
      i^{2})} }\gtrsim \f1{i\log i}  \;.
\end{equ}
Since this quantity isn't summable, $\sum_i \Delta_i = \infty$ and
we have found a contradiction to the assumption that $ T_{i} < T^{\mathrm{fin}} < \infty $ for all
$ i $.

\subsection{Energy estimate bounds}

The next sections are devoted to making rigorous the argument sketched above. We start by
obtaining the full energy
estimate, taking into consideration the rest terms which we
have ignored so far.

\begin{proposition}\label{prop:nrg-estimate}
   Fix $ t \mapsto \lambda_{t} $ as in
   Definition~\ref{def:stop-time-and-lambda} with $ \mf{a}= 3 $. There exists a
$ k \in \NN $, and a $ \kappa_{0} > 0 $ such that for some $ C > 0 $ and 
   all $ \kappa \in (0, \kappa_{0}  ) $ we can estimate uniformly over
   $ i \in \NN $ and $ t \in [T_{i} , T_{i +1}) $:
   \begin{equs}
      \partial_{t} \| w^{\mL} \|^{2} \leqslant  & - \| w^{\mL}
      \|_{H^{1}}^{2} +  2\langle w^{\mL},
      \mA_{t}^{\lambda_{t}} w^{\mL} \rangle  
+  2 \mf{r}_{\lambda} (t) \|
   w^{\mL} \|^{2} +    C  \lambda_{t}^{\frac{1}{3}}
(\mathbf{N}_{t}^{\kappa})^{k} \cdot \| w^{\mL} \|_{H^{1- \frac{3}{2}
\kappa}}\\
& +  C (\mathbf{N}_{t}^{\kappa})^{k} (\| w^{\mL} \|_{H^{1 -
\frac{3}{2} \kappa}}+ \| w^{\mL} \|_{H^{1- \frac{3}{2}
\kappa}}^{2})   \;,
   \end{equs}
   with $ \mf{r}_{\lambda} (t) $ defined by \eqref{eqn:ren-const}.
\end{proposition}

\begin{remark}\label{rem:on-a}
   The value $ \mf{a} = 3 $ is arbitrary: for example the calculations
   below allow for any $ \mf{a} \in (2, \infty) $. Our choice $
   \mf{a} = 3 $ guarantees that $ \lambda^{\frac{1}{3}}_{t} \lesssim 1 + \|
   w_{t} \| $, see also Remarks~\ref{rem:a} and~\ref{rem:factor-3}.
\end{remark}

\begin{proof}
   This estimate follows from the bound \eqref{e:lap-split} for the term in
   \eqref{eqn:bd-operator}, together with
   Lemmas~\ref{lem:interp},~\ref{lem:bd-rest}, and~\ref{lem:bd-4} below for 
   \eqref{eqn:bd-interpolation}--\eqref{eqn:bd-4}. The regularity $ 1 -
\frac{3}{2} \kappa $ is the worst one appearing in all estimates, and comes
from Lemma~\ref{lem:bd-4}.
\end{proof}
In the rest of this section we collect the bounds that lead to the energy estimate
in Proposition~\ref{prop:nrg-estimate}. We start with a bound on the term
\eqref{eqn:bd-interpolation}.

\begin{lemma}\label{lem:interp}
   Fix $ \lambda_{t} $ as in Definition~\ref{def:stop-time-and-lambda} for any $
   \mf{a} \in [2, \infty) $. There exists a $ \kappa_{0}> 0 $ and an
   $ \eta  \in (0, 1 - \kappa_{0}) $ such that for
   all $ \kappa < \kappa_{0} $ we have the bound
   \begin{equ}
      \langle w^{\mL},  \div ( 2
   (\mH_{\lambda_{t}} X) \sotimes w^{\mL} - 2
   (\mH_{\lambda_{t}} X ) \rpara w^{\mL} ) \rangle \lesssim
   \mathbf{N}_{t}^{\kappa} \| w^{\mL} \|_{H^{\eta}}^{2} 
   \end{equ}
\end{lemma}

\begin{proof}
   We can bound for any $ \eta \in (0, 1)  $
   \begin{equ}
      \|(\mH_{\lambda_{t}} X) \reso 
      w^{\mL} +  (\mH_{\lambda_{t}} X ) \para w^{\mL}
      \|_{H^{\eta - \kappa }} \lesssim \mathbf{N}_{t}^{\kappa} \| w^{\mL}
      \|_{H^{\eta}}\;,
   \end{equ}
   so that overall
   \begin{equ}
      \langle w^{\mL},  \div ( 2 (\mH_{\lambda_{t}} X) \sotimes
      w^{\mL} - 2 (\mH_{\lambda_{t}} X ) \rpara w^{\mL} ) \rangle
      \lesssim \mathbf{N}_{t}^{\kappa} \| w^{\mL} \|_{H^{\eta}}^{2}  \;,
   \end{equ}
provided that 
$\eta \geqslant \frac{1}{2} (1+ \kappa)$,
   which is the desired bound.
\end{proof} 
Next we pass to an estimate of \eqref{eqn:bd-3} and the first term of \eqref{eqn:bd-4}.

\begin{lemma}\label{lem:bd-rest}
   Fix $ \lambda_{t} $ as in Definition~\ref{def:stop-time-and-lambda} for any $
   \mf{a} \in (2, \infty) $. There exists a $ \kappa_{0}(\mf{a}) > 0 $ and an
   $ \eta (\mf{a}) \in (0, 1 - \kappa_{0}) $ such that for
   all $ \kappa < \kappa_{0} (\mf{a}) $ we have the bounds
   \begin{equs}
      \langle w^{\mL},    \div ( 2 X \sotimes w^{\mH} - 2
   (\mH_{\lambda_{t}} X ) \rpara w^{\mH} ) \rangle & \lesssim \| w^{\mL}
   \|_{H^{\eta}} \lambda_{t}^{\frac{1}{3}} ( \mathbf{N}_{t}^{\kappa}
   )^{2}   \;, \\ 
   \langle w^{\mL}, \div
   (w^{\otimes 2} )\rangle & \lesssim \| w^{\mL} \|_{H^{\eta}} (\|
   w^{\mL} \|_{H^{\eta}} + \mathbf{N}_{t}^{\kappa} ) (\mathbf{N}_{t}^{\kappa})^{2} \;.
   \end{equs}
\end{lemma}

\begin{remark}\label{rem:factor-3}
   The factor $ \lambda^{\frac{1}{3}}_{t} $ above could be replaced by 
   $ \lambda^{q}_{t} $ for an arbitrary $ q < 1/2$. Eventually we will
   need $ \lambda_{t}^{q} \lesssim 1 + \| w_{t}\|$, which is the case
   if $ \mf{a} \leqslant q^{-1}  $. Since any $ 2 < \mf{ a} \leqslant 3 $ will be
   sufficient for our needs we have fixed $ q =1/3 $.
\end{remark}

\begin{proof}
   \textit{First bound.} By Lemma~\ref{lem:high-feq-ref} with $ \delta =
   1/2 $ and the assumption $ \mf{a} \geqslant 2 $ one has
   \begin{equ}
      \|  (\mH_{\lambda_{t}}  X ) \reso w^{\mH} + 
     (\mH_{\lambda_{t}} X ) \para w^{\mH}  \|_{H^{1/2 - 3 \kappa}}
      \lesssim \mathbf{N}_{t}^{\kappa} \| w^{\mH}  \|_{H^{1/2 - 2 \kappa}}  \lesssim
     (\mathbf{N}_{t}^{\kappa})^{2} \;.
  \end{equ}
  Applying in addition Lemma~\ref{lem:low-freq} we obtain
  \begin{equ}
     \| (\mL_{\lambda_{t}} X ) \sotimes w^{\mH}  \|_{H^{1/3 -  \kappa}} 
     \lesssim \| \mL_{\lambda_{t}} X \|_{\mC^{1/3 - \kappa}} \|
     w^{\mH} \|_{H^{1/2 - 2 \kappa}} \lesssim
     \lambda_{t}^{\frac{1}{3}} ( \mathbf{N}_{t}^{\kappa} )^{2}\;,
   \end{equ}
   assuming that $ \kappa_{0} > 0 $ is sufficiently small so that $ 5/6 - 3
\kappa > 0 $.
   From the latter two estimates we can deduce as desired
   \begin{equ}
      \langle w^{\mL},  \div ( 2 X \sotimes w^{\mH} - 2
      (\mH_{\lambda_{t}} X ) \rpara w^{\mH} ) \rangle \lesssim \| w^{\mL}
      \|_{H^{\eta}}  \lambda_{t}^{\frac{1}{3}} (
      \mathbf{N}_{t}^{\kappa})^{2} \;,
   \end{equ}
   for any $ \eta \in ( 2/3+ \kappa, 1)  $, assuming that $ \kappa_{0} $ is small
so that $ 1/2 - 3 \kappa \geqslant 1/3 - \kappa $.

   \textit{Second bound.} Here, since $ w^{\otimes 2} =
   ( w^{\mL})^{\otimes 2} + ( w^{\mH})^{\otimes 2} + 2
   w^{\mL} \sotimes w^{\mH}$ and $
   w^{\mL} $ is divergence free, we find
   \begin{equ}
      \langle w^{\mL}, \div ( w^{\otimes 2}) \rangle  = \langle
      w^{\mL} , \div ( (w^{\mH} )^{\otimes 2} + 2
      w^{\mL} \sotimes w^{\mH})\rangle  \;.
   \end{equ}
   For the term involving $ (w^{\mH} )^{ \otimes 2} $ we
   estimate by Lemma~\ref{lem:parap} and Sobolev embeddings in dimension $ d =2 $
   \begin{equ}
      \| (w^{\mH})^{\otimes 2} \|_{H^{1/4 - 3 \kappa}} \lesssim \|
      w^{\mH} \|_{\mC_{4}^{1/4 - 2 \kappa}}^{2} \lesssim \|
      w^{\mH} \|^{2}_{\mC_{2}^{3/4- 2 \kappa}} \lesssim \|
      w^{\mH} \|^{2}_{H^{3/4 - 2 \kappa}} \;.
   \end{equ}
   To close this first bound we apply Lemma~\ref{lem:high-feq-ref} with $ \delta =
   1/4 $ in order to deduce
   \begin{equ}[e:boundwH]
      \| (w^{\mH})^{\otimes 2} \|_{H^{1/4 - 3 \kappa}} \lesssim \|
      w \| ( \mathbf{N}_{t}^{\kappa})^{2} \;,
   \end{equ}
   since $ \mf{a} \geqslant 2 $. Therefore we have obtained
   \begin{equ}
      \langle w^{\mL}, \div ( (w^{\mH})^{\otimes 2} ) \rangle
      \lesssim \| w^{\mL} \|_{H^{3/4 + 3 \kappa}} \| w \|
      (\mathbf{N}_{t}^{\kappa})^{2} \lesssim \| w^{\mL} \|_{H^{3/4 + 3 \kappa}} (\|
      w^{\mL} \| + \mathbf{N}_{t}^{\kappa} ) ( \mathbf{N}_{t}^{\kappa})^{2}\;,
   \end{equ}
   where in the last step we used once again Lemma~\ref{lem:high-feq-ref} with
   $ \delta = 1/2 $.

   For the term involving $ w^{\mL} \sotimes w^{\mH} $ we
   proceed similarly, only this time we make use of the fact that $
   \mf{a} > 2 $ (with a strict inequality!). We use
   Lemma~\ref{lem:high-feq-ref} with $ \delta = 1/ \mf{a} $ and Sobolev
   embeddings to bound
   \begin{equ}
      \| w^{\mH} \|_{\mC_{4}^{1/2 - 1/ \mf{a} - 2 \kappa}} \lesssim \|
      w^{\mH} \|_{H^{1 - 1/ \mf{a}- 2 \kappa} } \lesssim
      \mathbf{N}_{t}^{\kappa} \;, \qquad \| w^{\mL} \|_{\mC_{4}^{1/2 - 1/
      \mf{a}- \kappa}}
      \lesssim \| w^{\mL} \|_{H^{1 -1/ \mf{a}- \kappa}} \;.
   \end{equ}
   Now, if $ 1/2 - 1/ \mf{a} - \kappa > 0 $ -- which is the case since $
   \mf{a} > 2 $, assuming $ \kappa_{0}(\mf{a}) $ is sufficiently small -- we conclude that
   \begin{equ}
      \| \div( w^{\mL} \sotimes w^{\mH}) \|_{H^{- 1/2 - 1/
      \mf{a} - 2 \kappa}} \lesssim \| w^{\mL} \|_{H^{1 - 1/
      \mf{a} - \kappa}} \mathbf{N}_{t}^{\kappa}  \;.
   \end{equ}
   Therefore
   \begin{equ}
      \langle w^{\mL}, \div ( w^{\mL} \sotimes w^{\mH})
      \rangle \lesssim \| w^{\mL} \|_{H^{1/2 + 1/ \mf{a} + 2 \kappa}} \|
      w^{\mL} \|_{H^{1 - 1/ \mf{a}- \kappa}} \mathbf{N}_{t}^{\kappa} \lesssim \|
      w^{\mL} \|_{H^{\eta}}^{2} \mathbf{N}_{t}^{\kappa} \;,
   \end{equ}
   provided $ \kappa >0  $ is sufficiently small with respect to $
   \mf{a} $ and $ \eta \geqslant (1/2 + 1/ \mf{a} + 2 \kappa) \vee (1 - 1/
   \mf{a} - \kappa ) $. The proof is complete.
\end{proof} 
We are left with the last terms in \eqref{eqn:bd-4}.

\begin{lemma}\label{lem:bd-4}
Fix $ \lambda_{t} $ as in Definition~\ref{def:stop-time-and-lambda} for any $
   \mf{a} \in (2, \infty) $.
   There exists a $ \kappa_{0}(\mf{a}) > 0 $ and an $ \eta (\mf{a}) \in (0, 1 -
   \kappa_{0}) $ such that for any $ \kappa \in (0, \kappa_{0} (\mf{a})) $ 
   \begin{equs}
      \langle w^{\mL},  &\div
   (2 Y \sotimes w   +
   C^{\para }(w, Q^{\mH}) +  Y^{\otimes 2})\rangle \\
   & \lesssim  \|
   w^{\mL} \|_{H^{1 -  \frac{3}{2} \kappa}} (\| w^{\mL} \|_{H^{2 \kappa}} +
   \mathbf{N}_{t}^{\kappa} ) \mathbf{N}_{t}^{\kappa} + (\|
         w^{\mL} \|^{2}_{H^{\eta}} + \| w^{\mL}
      \|_{H^{\eta}} \mathbf{N}_{t}^{\kappa}  ) (\mathbf{N}_{t}^{\kappa} )^{2} \;.
   \end{equs}
\end{lemma}

\begin{proof}
   \textit{The $ 2 Y \sotimes w + Y^{\otimes 2}  $ term.} We can estimate via
   Lemma~\ref{lem:parap} and Lemma~\ref{lem:high-feq-ref}
   with $ \delta = 1/2 $, since $ \mf{a} \geqslant 2 $:
   \begin{equ}
      \| Y \sotimes w \|_{H^{\frac{7}{4} \kappa}} \lesssim \| Y 
      \sotimes w  \|_{\mC^{2 \kappa}_{2}} \lesssim \| Y  \|_{\mC^{2 \kappa}}
      \| w \|_{H^{2 \kappa}} \lesssim \mathbf{N}_{t}^{\kappa} ( \|
      w^{\mL} \|_{H^{2 \kappa}} + \mathbf{N}_{t}^{\kappa}) \;.
   \end{equ}
   Hence, since $ \frac{7}{4} \kappa - \frac{3}{2} \kappa  > 0 $ 
   \begin{equ}
      \langle w^{\mL},  \div ( Y \sotimes w)
      \rangle \lesssim \| w^{\mL} \|_{H^{1 -  \frac{3}{2} \kappa}} (\|
      w^{\mL} \|_{H^{2 \kappa}} + \mathbf{N}_{t}^{\kappa})
      \mathbf{N}_{t}^{\kappa} \;.
   \end{equ}
   Similarly for the $ Y^{\otimes 2} $ term 
   \begin{equ}
      \langle w^{\mL}, \div ( Y^{\otimes 2}) \rangle \lesssim
      \| w^{\mL} \|_{H^{1-2\kappa}} ( \mathbf{N}_{t}^{\kappa})^{2} \;.
   \end{equ}

\textit{The commutator term.} 
   Here by definition we have
   \begin{equ}
      C^{\para}(w, Q^{\mH}) = ( (\partial_{t}- \Delta) w)  \para
      Q^{\mH} + \mathrm{Tr} [ (\nabla w) \para (\nabla Q^{\mH}) ] \;.
   \end{equ}
   As for the first term, from \eqref{eqn:w}
   \begin{equ}[e:support]
      ( (\partial_{t}- \Delta) w)  \para
      Q^{\mH} = \left[ \mathbf{P} \div (w^{\otimes 2} + D
      \sotimes w + Y^{\otimes 2}) \right] \para Q^{\mH} \;.
   \end{equ}
   Let us start with the quadratic part, which has the worse homogeneity. By
   Sobolev embeddings we obtain
$      \| w^{\otimes 2} \|_{L^{2}} \lesssim \| w
      \|_{L^{4}}^{2} \lesssim \| w
      \|^{2}_{H^{1/2}}$
   so that for any $ \gamma > 0 $, by Lemma~\ref{lem:parap} and
\eqref{e:QH-bd}
   \begin{equs}
      \| [ \mathbf{P} \div ( w^{\otimes 2})] \para Q^{\mH}
      \|_{H^{1 - 2 \kappa - \gamma}} & \lesssim \| w^{ \otimes 2}
      \|_{L^{2}} \| Q^{\mH}  \|_{\mC^{2 - \kappa - \gamma}}\\
      & \lesssim \| w \|_{H^{1/2}}^{2} \| Q^{\mH} \|_{\mC^{2 - \kappa -
      \gamma}} \lesssim \| w \|_{H^{1/2}}^{2}
      \lambda_{t}^{- \gamma} \mathbf{N}_{t}^{\kappa} \;.
   \end{equs}
   Therefore we find
   \begin{equs}
      \langle w^{\mL}, \div \left( [ \mathbf{P} \div
      ( w^{\otimes 2} )] \para Q^{\mH} \right) \rangle & \lesssim \|
      w^{\mL} \|_{H^{2 \kappa + \gamma}} \| [ \mathbf{P} \div ( w^{\otimes 2})] \para Q^{\mH}
      \|_{H^{1 - 2 \kappa - \gamma}} \\
      & \lesssim\| w^{\mL} \|_{H^{2 \kappa + \gamma}}\| w
      \|_{H^{1/2 + \kappa}}^{2} \lambda_{t}^{- \gamma}
      \mathbf{N}_{t}^{\kappa} \;.
   \end{equs}
Let us now use the decomposition $ w = w^{\mH} + w^{\mL} $, so that we can
further bound
   \begin{equs}
      \langle w^{\mL},  \div & \left( [ \mathbf{P} \div
      ( w^{\otimes 2} )] \para Q^{\mH} \right) \rangle \\
      & \lesssim \|
      w^{\mL} \|_{H^{\gamma+ 2 \kappa }} (\|
      w^{\mL} \|^{2}_{H^{1/2+ \kappa}}+ \| w^{\mH} \|^{2}_{H^{1/2 + \kappa}} )(1+\| w \|)^{-
   \gamma \mf{a}} \mathbf{N}_{t}^{\kappa} \\
      & \lesssim \| w^{\mL} \|_{H^{\gamma + 2 \kappa }} (\|
      w^{\mL} \|^{2}_{H^{1/2+ \kappa}}+ (\mathbf{N}_{t}^{\kappa})^{2} )(1 + \| w \|)^{-
      \gamma \mf{a}} \mathbf{N}_{t}^{\kappa} \;,
   \end{equs}
   where in the last step we used Lemma~\ref{lem:high-feq-ref} with $ \delta =
   1/ \mf{a} $ together with
   the assumption $ \mf{a} > 2 $ so that, provided $ \kappa_{0}(\mf{a}) > 0 $ is sufficiently
   small: $$ \| w^{\mH}
   \|_{H^{1/2 + 2 \kappa}} \leqslant \| w^{\mH} \|_{H^{1 - 1/
   \mf{a} - 2\kappa}} \lesssim \mathbf{N}_{t}^{\kappa} \;.$$ 
Next, we can use interpolation to
   bound, for any $ \eta \geqslant (1/2 + 2 \kappa) \wedge (\gamma + 2 \kappa) $:
   \begin{equs}
      \| w^{\mL} \|_{H^{1/2 +  \kappa}} & \lesssim \| w^{\mL}
      \|_{H^{\eta}}^{p (\eta, \kappa)} \| w^{\mL} \|^{1 - p (\eta,
      \kappa)} \;, \qquad p (\eta, \kappa) = \frac{1/2 + \kappa}{\eta} \in
      (0, 1) \;, \\
      \| w^{\mL} \|_{H^{\gamma + 2 \kappa}} & \lesssim \| w^{\mL}
      \|_{H^{\eta}}^{q (\eta, \kappa)} \| w^{\mL} \|^{1 - q (\eta,
      \kappa)} \;, \qquad q (\eta, \kappa) = \frac{\gamma + 2\kappa}{\eta} \in
      (0, 1) \;.
   \end{equs}
Hence we obtain
\begin{equs}
 \| w^{\mL} \|_{H^{\gamma + 2 \kappa }} \|
      w^{\mL} \|^{2}_{H^{1/2+ \kappa}}& (1 + \| w \|)^{-
      \gamma \mf{a}} \\
&  \lesssim \| w^{\mL} \|_{H^{\eta}}^{q} \|
      w^{\mL} \|^{2p}_{H^{\eta}} \| w^{\mL} \|^{3 - 2 p - q}(1 + \| w \|)^{-
      \gamma \mf{a}}  \;.
\end{equs}
To eventually find a useful estimate we must pick $ \eta, \gamma $ such that
for all \(\kappa\) small
\begin{equs}
2 p(\eta, \kappa)+ q(\eta, \kappa) \leqslant 2 \;, \qquad 3 -2 p(\eta,
\kappa)-q(\eta, \kappa)
\leqslant \gamma \mf{a} \;.
\end{equs}
For example, if we fix $ \gamma = 1/2 $ and $ \eta = 3/4 + 2 \kappa $, then the
first inequality is satisfied:
\begin{equs}
2 p + q = \frac{1 + \gamma + 4 \kappa }{\eta} = \frac{3/2 + 4 \kappa}{3/4 + 2
\kappa} = 2 \frac{3/2 + 4 \kappa}{3/2 + 4 \kappa} = 2\;,
\end{equs}
and the second inequality as well, since for $ \mf{a} \geqslant 2 $:
\begin{equs}
3 - 2 p - q =  1 \leqslant \gamma \mf{a} \;.
\end{equs}
With such a choice, we can finally bound
\begin{equs}
  \langle w^{\mL},  \div \left( [ \mathbf{P} \div
      ( w^{\otimes 2} )] \para Q^{\mH} \right) \rangle  \lesssim \|
         w^{\mL} \|^{2}_{H^{\eta}} \mathbf{N}_{t}^{\kappa} + \| w^{\mL}
\|_{H^{\eta}} (\mathbf{N}^{\kappa}_{t})^{3} \;,
\end{equs}
   which is of the required order. 
Now we can proceed to the last two terms in
   \eqref{e:support}. First of all, since $ D = 2(X + Y) $, we have
   \begin{equs}
      \| D \sotimes w + Y^{\otimes 2}\|_{H^{- \kappa}} & \lesssim \| D
      \|_{\mC^{- \kappa}} \| w \|_{H^{2 \kappa}} +
(\mathbf{N}_{t}^{\kappa})^{2} \\
& \lesssim
      \mathbf{N}_{t}^{\kappa} \| w
      \|_{H^{2 \kappa}} + (\mathbf{N}_{t}^{\kappa})^{2} \;.
   \end{equs}
   Therefore
   \begin{equs}
      \| [ \mathbf{P} \div (D \sotimes w + Y^{\otimes 2})] \para
      Q^{\mH}  \|_{H^{1 - 2 \kappa}} & \lesssim \| D \sotimes w + Y^{\otimes
      2}\|_{H^{- \kappa}} \| Q^{\mH} \|_{H^{2 - \kappa}}\\
      & \lesssim ( \mathbf{N}_{t}^{\kappa})^{2} \| w \|_{H^{2 \kappa}} +
      (\mathbf{N}_{t}^{\kappa})^{3} \;.
   \end{equs}
   Plugging this into the desired inner product we conclude
   \begin{equs}
      \langle w^{\mL},  \div \left( [ \mathbf{P} \div (D
      \sotimes w + Y^{\otimes 2})] \para Q^{\mH} \right) \rangle & \lesssim
      \| w^{\mL} \|_{H^{ 2\kappa}} ( (\mathbf{N}_{t}^{\kappa})^{2} \| w
      \|_{H^{2 \kappa}} + (\mathbf{N}_{t}^{\kappa})^{3}) \\
      & \lesssim \| w^{\mL} \|_{H^{ 2\kappa}} ( (\mathbf{N}_{t}^{\kappa})^{2} \|
      w^{\mL} \|_{H^{2 \kappa}} + (\mathbf{N}_{t}^{\kappa})^{3}) \;,
   \end{equs}
   where in the last step we made use of Lemma~\ref{lem:high-feq-ref}. This
   concludes the proof.
\end{proof}

\section{Global solutions}\label{sec:global-solutions}

As in the previous section, we work under Assumption~\ref{assu:L2IC} and assume
that the initial condition $ u_{0} $ to \eqref{eqn:w} satisfies
\begin{equ}
   u_{0} \in L^{2} \cap \mC^{-1 + \kappa}\;,
\end{equ}
for some $ \kappa > 0 $.
The objective of this section is to build on Proposition~\ref{prop:nrg-estimate} to obtain
global well-posedness for \eqref{eqn:daPrato-Debussche}. 
The first step is to apply some
interpolation inequalities to obtain a bound on the distance between the
successive stopping times $ T_{i +1} - T_{i} $. We start with a corollary of
Proposition~\ref{prop:nrg-estimate}.
\begin{corollary}\label{cor:unifrm-bd}
   In the setting of Proposition~\ref{prop:nrg-estimate}, for some $ \kappa_{0}
   > 0 $ there exists a constant $ C_{1} >0 $ and increasing
   continuous maps $C_{2}, C_{3} \colon \RR_{+} \to (0, \infty) $ such that for
   all $ \kappa \leqslant \kappa_{0} $ we can
   estimate uniformly over $ i \in \NN, i \geqslant i_{0} $ and $ t \in [T_{i}, T_{i+1}) $
   \begin{equ}
      \partial_{t} \| w^{\mL}_{t} \|^{2}+ \frac{1}{2} \| w^{\mL}_{t} \|_{H^{1}}^{2} 
      \leqslant C_{1} \log{
      (\lambda_{t})} \| w^{\mL}_{t} \|^{2} + C_{2} (\mathbf{N}^{\kappa}_{t}) \|
      w^{\mL}_{t} \|^{2} + C_{3}(\mathbf{N}^{\kappa}_{t}) \;,
   \end{equ}
   In particular, we can estimate
   \begin{equs}
      \sup_{T_{i} \leqslant t < T_{i+1}} &\| w^{\mL}_{t} \|^{2} + 
      \frac{1}{2} \int_{T_{i}}^{T_{i+1}} \|
      w^{\mL}_{s} \|_{H^{1}}^{2} \ud s \\
      & \leqslant \left( \| w^{\mL}_{T_{i}} \|^{2} +  C_{3} (
      \mathbf{N}^{\kappa}_{T_{i+1}}) \right) \cdot \exp \{
      (T_{i+1} - T_{i}) [C_{2} ( \mathbf{N}^{\kappa }_{T_{i +1}}) + C_{1} \log{(\lambda_{T_{i}})}] \}\;.
   \end{equs}
\end{corollary}

\begin{proof}

From Proposition~\ref{prop:nrg-estimate} we have that for some $ C >0 $
   \begin{equs}
      \partial_{t} \| w^{\mL} \|^{2} \leqslant  & - \| w^{\mL}
      \|_{H^{1}}^{2} +  2\langle w^{\mL},
      \mA_{t}^{\lambda_{t}} w^{\mL} \rangle  
+  2 \mf{r}_{\lambda} (t) \|
   w^{\mL} \|^{2} +    C  \lambda_{t}^{\frac{1}{3}}
(\mathbf{N}_{t}^{\kappa})^{k} \cdot \| w^{\mL} \|_{H^{1- \frac{3}{2}
\kappa}}\\
& +  C (\mathbf{N}_{t}^{\kappa})^{k} (\| w^{\mL} \|_{H^{1 -
\frac{3}{2} \kappa}}+ \| w^{\mL} \|_{H^{1- \frac{3}{2}
\kappa}}^{2})   \;,
   \end{equs}
Regarding the quadratic form associated to $ \mA^{\lambda_{t}}_{t}$, it follows from 
   Proposition~\ref{prop:def-A} below that there exists an $ \mathbf{m} 
   (\mathbf{N}^{\kappa}_{t}) $ such that 
   \begin{equ}
      \langle w^{\mL}, \mA^{\lambda_{t}}_{t} w^{\mL} \rangle  \leqslant
      \mathbf{m} (\mathbf{N}^{\kappa}_{t}) \| w^{\mL} \|^{2} \;.
   \end{equ}
Regarding $ \mf{r}_{\lambda}(t) $ we find for some $ c >0 $ that $
\mf{r}_{\lambda} (t) \leqslant c \cdot \log{(\lambda_{t})} $, see
\eqref{eqn:ren-const}. For the term
involving \( \lambda_{t} \) we estimate $\lambda_{t}^{\frac{1}{3}} \lesssim (1+
\| w_{t} \|) $ for $ t \in [T_{i}, T_{i+1})$ since we have assumed that $
\mf{a} = 3 $.
Hence overall we find $ C_{1}, C_{2}, C_{3} $ as in the statement of the
corollary, such that
\begin{equs}
   \partial_{t} \| w_{t}^{\mL} \|^{2} 
   \leqslant - \frac{1}{2}  \| w^{\mL} \|_{H^{1}}^{2}+ (  C_{1}
\log{(\lambda_{t})}+ C_{2}( \mathbf{N}^{\kappa}_{t}) ) \| w^{\mL} \|^{2}  +
C_{3} (\mathbf{N}_{t}^{\kappa}) \;.
\end{equs}
Here we repeatedly use interpolation and Young's inequality for products so
that for any $
\eta, \ve \in (0, 1) $ there exists a $ C(\ve, \eta)> 0 $ such that
\begin{equ}
   \| w^{\mL} \|_{H^{\eta}}^{2} \leqslant  \ve \| w^{\mL} \|_{H^{1}}^{2} +
   C (\ve, \eta) \| w^{\mL} \|^{2} \;.
\end{equ}

As for the second estimate, we find for any $ t \in [T_{i}, T_{i+1}) $ and $
\mu = C_{2}(\mathbf{N}^{\kappa}_{T_{i+1}}) + C_{1}
\log{(\lambda_{T_{i}})} $:
\begin{equ}
   \| w^{\mL}_{t} \|^{2} \leqslant  e^{(t - T_{i}) \mu} \| w^{\mL}_{T_{i}}
   \|^{2}
    +  \int_{T_{i}}^{t}   - \frac{1}{2} e^{(t - s) \mu } \| w^{\mL}_{s} \|_{H^{1}}^{2}
    + e^{(t - s) \mu } C_{3}(\mathbf{N}^{\kappa}_{t})\ud s \;,
\end{equ}
so that
\begin{equs}
   \| w^{\mL}_{t} \|^{2}+  \frac{1}{2} \int_{T_{i}}^{t}
   \| w^{\mL}_{s} \|_{H^{1}}^{2}\ud s 
   &\leqslant  \| w^{\mL}_{t} \|^{2}+ \frac{1}{2} \int_{T_{i}}^{t} e^{(t - s) \mu }
   \| w^{\mL}_{s} \|_{H^{1}}^{2} \ud s \\
   &\leqslant e^{(T_{i+1} - T_{i}) \mu} \| w^{\mL}_{T_{i}}
   \|^{2} + C_{3} ( \mathbf{N}^{\kappa}_{T_{i+1}}) \int_{0}^{t} e^{(t -s)
   \mu} \ud s \\
   &\leqslant e^{(T_{i+1} - T_{i}) \mu} \| w^{\mL}_{T_{i}}
   \|^{2} + C_{3} ( \mathbf{N}^{\kappa}_{T_{i+1}}) \mu^{-1}e^{(T_{i+1} -
   T_{i}) \mu }\;,
\end{equs}
which implies the desired result.
\end{proof}
To complete the $ L^{2} $ estimate, we must control the jump of the norm at the
stopping times $ T_{i} $. 
\begin{lemma}\label{lem:interval}
   In the setting of Corollary~\ref{cor:unifrm-bd}, consider $ i \in
   \NN_{+} $ such that $ i \geqslant i_{0} (u_{0}) $, with $ i_{0}
(u_{0}) $ as in \eqref{e:i0}, and fix $ t > 0 $. Then if $ T_{i+1} <
T^{\mathrm{fin}} \wedge t $ there exists a constant $ C (\mathbf{N}_{t}^{\kappa}) > 0 $ such that
   \begin{equ}
      T_{i+1} - T_{i} \geqslant \frac{1}{C (\mathbf{N}_{t}^{\kappa})(1 + \log{(1 +i)}
      )} \cdot \log{ \left( \frac{i^{2} + 2 i - C(\mathbf{N}_{t}^{\kappa})}{
      i^{2}+ C(\mathbf{N}_{t}^{\kappa})}  \right) } \;.
   \end{equ}
\end{lemma}

\begin{proof}
   We can use Corollary~\ref{cor:unifrm-bd} to bound
\begin{equ}
   T_{i +1} - T_{i} \geqslant \frac{1}{C_{2} (\mathbf{N}_{t}^{\kappa}) +
   \log{(\lambda_{T_{i}})}} \cdot \log{ \left( \frac{\|
   w^{\mL}_{T_{i+1}-} \|^{2}}{ \| w^{\mL}_{T_{i}} \|^{2} +
   C_{3} (\mathbf{N}_{t}^{\kappa})} \right)} \;.
\end{equ}
Now, since $ T_{i +1} < T^{\mathrm{fin}} $, from the definition of the stopping
time we have for some $ c > 1 $
\begin{equs}
   \| w^{\mL}_{T_{i+1}-} \| & \geqslant \| w_{T_{i+1}} \| - \|
   w^{\mH}_{T_{i+1} -} \| \geqslant (i+1) - c \mathbf{N}_{t}^{\kappa}(i +1)^{-1}  \;, \\
   \| w^{\mL}_{T_{i}} \| & \leqslant \| w_{T_{i}} \|+ \| w
   ^{\mH}_{T_{i}} \|\leqslant i + c \mathbf{N}_{t}^{\kappa} i^{-1}  \;,
\end{equs}
by Lemma~\ref{lem:high-feq-ref} with $ \delta = 2/ \mf{a} \in (0, 1 - 2\kappa) $
(since $ \mf{a} > 2 $, for $ \kappa > 0 $ sufficiently small).
Recall here that since $\lambda_t$ jumps at $t = T_i$ by \eqref{e:def-lambda}, we also have jumps in
the definitions of $w^\mL$, see \eqref{eqn:def-phi-high}.
Hence we obtain that 
\begin{equ}
   T_{i +1} - T_{i} \geqslant \frac{1}{C_{2} (\mathbf{N}_{t}^{\kappa}) +
   \mf{a}\log{(1+i)}} \cdot \log{ \left( \frac{ (i + 1)^{2} - 2
         c \mathbf{N}_{t}^{\kappa} }{ i^{2} + 2 c \mathbf{N}_{t}^{\kappa}+ c^{2}
   (\mathbf{N}_{t}^{\kappa})^{2} i^{-2} + C_{3} (\mathbf{N}_{t}^{\kappa})  } \right)} \;,
\end{equ}
from which the result follows.
\end{proof}
The previous lemma gives us a control on the explosion time of the $
L^{2} $ norm. Next we show that if $ T^{\mathrm{fin}} < \infty $, then
$ \lim_{t \uparrow T^{\mathrm{fin}}} \| w^{\mL}_{t} \| = \infty $, meaning
that the explosion of the $ L^{2} $ norm is a necessary (and of course
sufficient) condition for finite-time blow-up. For this purpose we require
higher regularity estimates.

\begin{lemma}\label{lem:higher-reg}
   In the setting of Corollary~\ref{cor:unifrm-bd} there exists a $ \kappa_{0}
   > 0 $ such that the following holds for any $ \kappa \in (0,
   \kappa_{0})$ and $ \ve \in (0, \kappa) $ and. Fix any $ M>1, T > 0 $ such that
   \begin{equ}
      \| w^{\mL}_{0} \|_{H^{\ve}}^{2} + \sup_{0 \leqslant t \leqslant T \wedge
      T^{\mathrm{fin}}} \| w^{\mL}_{t}
      \|^{2} + \int_{0}^{T \wedge T^{\mathrm{fin}}} \| w^{\mL}_{t}
      \|_{H^{1}}^{2}\ud t \leqslant M \;.
   \end{equ}
   Then there exists a $ C (T, M, \mathbf{N}_{T}^{\kappa} ) \in (0, \infty) $ such that
   \begin{equ}
      \sup_{ 0 \leqslant t \leqslant T \wedge T^{\mathrm{fin}}} \|
      w^{\mL}_{t} \|_{H^{\ve}}^{2} \leqslant C(T, M, \mathbf{N}_{T}^{\kappa}) \;.
   \end{equ}
\end{lemma}

\begin{proof}

   To control the $ H^{\ve} $ norm we have to control  $
   \langle w^{\mL}, (- \Delta)^{\ve} w^{\mL} \rangle
   $. Here we find, as in \eqref{eqn:bd-operator}, \eqref{eqn:bd-interpolation},
\eqref{eqn:bd-3} and \eqref{eqn:bd-4}:
\minilab{e:fourTerms2}
   \begin{equs}
      \partial_{t} \langle w^{\mL}&, (- \Delta)^{\ve}
      w^{\mL} \rangle =  2 \langle (- \Delta)^{\ve}w^{\mL}, \partial_{t} w^{\mL}  \rangle \\
      &=  2\langle (- \Delta)^{\ve} w^{\mL}, \Delta
   w^{\mL} +  \div ( 2 ( \mL_{\lambda_{t}} X ) \sotimes
   w^{\mL}) \rangle \label{eqn:bd-operator-frac} \\ 
   & \quad +  2\langle (- \Delta)^{\ve}w^{\mL},  \div ( 2
   (\mH_{\lambda_{t}} X) \sotimes w^{\mL} - 2
   (\mH_{\lambda_{t}} X ) \rpara w^{\mL} ) \rangle
   \label{eqn:bd-interpolation-frac} \\
   & \quad + 2 \langle (- \Delta)^{\ve}w^{\mL},  \div ( 2 X \sotimes w^{\mH} - 2
   (\mH_{\lambda_{t}} X ) \rpara w^{\mH} )\rangle\label{eqn:bd-3-frac} \\ 
   & \quad + 2 \langle (- \Delta)^{\ve} w^{\mL},   \div
   (w^{\otimes 2} + 2 Y \sotimes w  +
   C^{\para }(w, Q^{\mH}) +  Y^{\otimes 2})\rangle \;.\quad 
   \label{eqn:bd-4-frac}
   \end{equs}
   We bound the right-hand side one term at the time. The value of the constants $
C >0$ may change from line to line. All calculations hold only for $
\kappa_{0}$ sufficiently small. For \eqref{eqn:bd-operator-frac}
   we have for any $ \delta \in (0, 1) $
   \begin{equs}
      \langle (- \Delta)^{\ve} w^{\mL}, \Delta w^{\mL} +
      \div ( 2 ( \mL_{\lambda_{t}} X ) \sotimes w^{\mL})
      \rangle
      & \leqslant - (1 - \delta)\| w^{\mL} \|_{H^{1 + \ve}}^{2} +
      C(M, \mathbf{N}_{T}^{\kappa}, \delta)\;,
   \end{equs}
where we used the estimate
\begin{equs}
\langle (- \Delta)^{\ve} w^{\mL}, \div ( 2 ( \mL_{\lambda_{t}} X ) \sotimes w^{\mL}) 
      \rangle & \leqslant C(M, \mathbf{N}_{T}^{\kappa}) \| w^{\mL}
\|_{H^{2 \ve}} \| w^{\mL} \|_{H^{1}} \\
& \leqslant \delta \| w^{\mL}
\|_{H^{1 + \ve}}^{2} + C(M, \mathbf{N}_{T}^{\kappa}, \delta) \;,
\end{equs}
for $ \ve $ sufficiently small.
   For \eqref{eqn:bd-interpolation-frac} we follow the proof of
   Lemma~\ref{lem:interp} to obtain for any $ \delta \in (0, 1) $
   \begin{equs}
      \langle (- \Delta)^{\ve} w^{\mL}, \div ( 2
      (\mH_{\lambda_{t}} X) & \sotimes w^{\mL} - 2
      (\mH_{\lambda_{t}} X ) \rpara w^{\mL} ) \rangle \\
      & \leqslant
      C(\mathbf{N}_{T}^{\kappa}) \| w^{\mL} \|_{H^{1/2 +3 \kappa + \ve}} \|
      w^{\mL} \|_{H^{1/2 - 2 \kappa + \ve }} \\
      & \leqslant \delta \| w^{\mL} \|_{H^{1 + \ve}}^{2} + C(M, \mathbf{N}_{T}^{\kappa}, \delta) \;,
   \end{equs}
   where the last bound follows by
   interpolation on the Sobolev norms.

   Next, following the proof of the first
   bound of Lemma~\ref{lem:bd-rest}, we obtain for \eqref{eqn:bd-3-frac} and
any choice of $ \delta \in (0, 1) $
   \begin{equs}
      \langle (- \Delta)^{\ve}w^{\mL},\div ( 2 X \sotimes w^{\mH} - 2
      (\mH_{\lambda_{t}} X ) \rpara w^{\mH} )\rangle 
      & \leqslant \| w^{\mL} \|_{H^{\eta+ 2\ve }} C(M,
      \mathbf{N}_{T}^{\kappa}) \\
& \leqslant \delta  \| w^{\mL} \|_{H^{1 + \ve }}^{2} +  C(M,
      \mathbf{N}_{T}^{\kappa}, \delta )\;,
   \end{equs}
for any $ \eta > 2/3+ \kappa $, making use of the estimate $ \lambda_{t} \lesssim
M^{\frac{\mf{a}}{2} } $ and since $ \eta + 2 \ve +
\kappa < 1 + \ve $ for $ \kappa_{0} $ sufficiently small.

   Finally, for \eqref{eqn:bd-4-frac} we start by estimating the cubic term. We
   can rewrite $ w^{\otimes 2} = ( w^{\mL})^{\otimes 2} +
   ( w^{\mH})^{\otimes 2} + 2 w^{\mL} \sotimes w^{\mH}$, and
   we will estimate one addend at a time.
   Starting with $ (w^{\mL})^{\otimes 2} $, we obtain 
   \begin{equs}
      \langle (- \Delta)^{\ve} w^{\mL},   \div
      ( (w^{\mL})^{\otimes 2}) \rangle &\leqslant C \| w^{\mL} \|_{H^{1 + \ve}}
      \| (w^{\mL})^{\otimes 2} \|_{H^{\ve}} \;.
   \end{equs}
Next, we can estimate via a Kato--Ponce type inequality (see for example
\cite[Thm~A.13]{KPV}) in dimension $ d =2 $ and by interpolation
\begin{equs}
\| (w^{\mL})^{\otimes 2} \|_{H^{\ve}} & \lesssim \| w^{\mL}
\|_{H^{\frac{1}{2} + \ve} } \|
w^{\mL} \|_{H^{\frac{1}{2}}}  \lesssim \| w^{\mL} \|^{\frac{1}{2}} \|
w^{\mL} \|^{\frac{1}{2}}_{H^{1}} \| w^{\mL} \|_{H^{\ve}}^{\frac{1}{2}
} \| w^{\mL} \|_{H^{1 + \ve}}^{\frac{1}{2}} \\
& \lesssim M  \|
w^{\mL} \|^{\frac{1}{2}}_{H^{1}} \| w^{\mL} \|_{H^{\ve}}^{\frac{1}{2}
} \| w^{\mL} \|_{H^{1 + \ve}}^{\frac{1}{2}}\;.
\end{equs}
Therefore we obtain that
\begin{equs}
\langle (- \Delta)^{\ve} w^{\mL},   \div
      ( (w^{\mL})^{\otimes 2}) \rangle & \leqslant C(M) \| w^{\mL}
\|_{H^{1 + \ve}}^{\frac{3}{2}} \|
w^{\mL} \|^{\frac{1}{2}}_{H^{1}} \| w^{\mL} \|_{H^{\ve}}^{\frac{1}{2} } \\
& \leqslant \delta \| w^{\mL} \|_{H^{1 + \ve}}^{2} + C(M, \delta) \|
w^{\mL} \|^{2}_{H^{1}} \| w^{\mL} \|_{H^{\ve}}^{2}\;,
\end{equs}
where we used Young's inequality with conjugate exponents $ 4/3 $ and $ 4
$.

   For the other two terms we follow the proof of Lemma~\ref{lem:bd-rest}. In
   particular, we use the bound \eqref{e:boundwH} to obtain
   \begin{equs}
      \langle ( - \Delta)^{\ve}  w^{\mL}, \div ( (w^{\mH})^{\otimes 2} ) \rangle
      & \leqslant C \| w^{\mL} \|_{H^{3/4 + 3 \kappa + 2 \ve}} \| w \|
      (\mathbf{N}_{t}^{\kappa})^{2} \\
      & \leqslant C(M, \mathbf{N}_{T}^{\kappa}) \| w^{\mL} \|_{H^{3/4
      + 3 \kappa + 2 \ve}}\\
      & \leqslant \delta \| w^{\mL} \|_{H^{1 +  \ve}} + C (M,
\mathbf{N}_{T}^{\kappa}, \delta ) \;,
   \end{equs}
since $ \ve \in (0, 1/6) $.
   As for the term involving $ w^{\mL} \sotimes w^{\mH}
   $, we follow once more the proof of Lemma~\ref{lem:bd-rest} to obtain
   \begin{equ}
      \| \div( w^{\mL} \sotimes w^{\mH}) \|_{H^{- 1/2 - 1/
      \mf{a} - 2 \kappa}} \lesssim \| w^{\mL} \|_{H^{1 - 1/
      \mf{a} - \kappa}} \mathbf{N}_{t}^{\kappa} \;,
   \end{equ}
   so that for some $ \eta \in (0, 1)$ (assuming that $
   \kappa_{0} $ is sufficiently small and since $ \mf{a} =3 $) and any $
   \delta \in (0, 1) $:
\begin{equs}
   \langle (- \Delta)^{\ve} w^{\mL}, \div ( w^{\mL} \sotimes w^{\mH})
      \rangle & \lesssim \| w^{\mL} \|_{H^{1/2 + 1/ \mf{a} + 2 \kappa + 2 \ve}} \|
      w^{\mL} \|_{H^{1 - 1/ \mf{a}- \kappa}} \mathbf{N}_{t}^{\kappa}\\
      & \leqslant C(\mathbf{N}_{T}^{\kappa})  \| w^{\mL} \|_{H^{\eta}}^{2} \\
      & \leqslant \delta \| w^{\mL} \|_{H^{1 + \ve}}^{2} +
      C( M,\mathbf{N}_{T}^{\kappa}, \delta)\;,
   \end{equs}
for $ \ve $ sufficiently small.

   To conclude our
   estimate for \eqref{eqn:bd-4-frac}, we have to bound
   \begin{equ}
    L \eqdef  \langle (- \Delta)^{\ve} w^{\mL}, \div ( 2 Y
      \sotimes w + C^{\para }(w, Q^{\mH}) + Y^{\otimes 2}) \rangle \;.
   \end{equ}
   Following the same steps as in Lemma~\ref{lem:bd-4}, we obtain
   \begin{equs}
L \leqslant C(M, \mathbf{N}_{T}^{\kappa}) (1 + \| w^{\mL} \|_{H^{1 + 2 \ve - \frac{3}{2}
      \kappa}}^{2}) 
      \leqslant \delta \| w^{\mL} \|_{H^{1 + \ve}}^{2} +  C(M,
      \mathbf{N}_{T}^{\kappa}, \delta) \;,
   \end{equs}
   since $ \ve \leqslant \kappa $, which is again of the desired order. 

Overall, choosing $ \delta \in
(0, 1)  $
   sufficiently small we have obtained
   \begin{equ}
      \partial_{t} \| w^{\mL} \|^{2}_{H^{\ve}} \leqslant
      C(M, T, \mathbf{N}_{T}^{\kappa}) (1 + \| w^{\mL} \|_{H^{1}}^{2} +  \|
      w^{\mL} \|^{2}_{H^{\ve}}\| w^{\mL} \|_{H^{1}}^{2}) \;.
   \end{equ}
Therefore, if we define $ g_{t} = 1 + \| w^{\mL}_{t} \|^{2}_{H^{\ve}} $, we
have
\begin{equs}
\partial_{t} g_{t} \leqslant C(M, T, \mathbf{N}_{T}^{\kappa}) (1 + \|
w^{\mL }_{t} \|_{H^{1}}^{2}) g_{t} \;,
\end{equs}
so that by Gronwall's inequality
\begin{equ}
g_{t} \leqslant g_{0} \exp \Big( C(M, T, \mathbf{N}_{T}^{\kappa})
\int_{0}^{T} (1 + \| w^{\mL}_{s} \|_{H^{1}}^{2}) \ud s \Big) \;,
\end{equ}
for all $ t \in [0, T] $, which is the desired bound.
\end{proof}
Finally, we can deduce our $ L^{2} $ blow-up criterion.

   \begin{corollary}\label{cor:explosion-condition}
      Under Assumption~\ref{assu:L2IC}, 
      if $ T^{\fin} < \infty $, then $ \limsup_{t\uparrow T^{\mathrm{fin}}} \|
      w_{t} \| = \infty$.
   \end{corollary}
   \begin{proof}
	It suffices to prove that $ T_{i} < T^{\mathrm{fin}} $ for all $ i \in
      \NN$.
      Note that under Assumption~\ref{assu:L2IC}, by
      Proposition~\ref{prop:wp-phi} for any $ \zeta < 1 - \kappa $ and $ 0 < t
      < T^{\mathrm{fin}} $ we have $ \| w^{\mL}_{t} \|_{H^{\zeta}} < \infty $.
      If by contradiction there exists an $ i_{\mathrm{fin}} \in \NN $ such that
      $ T_{i} = T^{\mathrm{fin}} $ for all $ i \geqslant i_{\mathrm{fin}}  $,
      then by Corollary~\ref{cor:unifrm-bd} and Lemma~\ref{lem:higher-reg} we
      would find a $ C( T^{\mathrm{fin}},
      \mathbf{N}^{\kappa}_{T^{\mathrm{fin}}}) >0 $ and an $ \ve> 0 $ such that
      \begin{equ}
         \sup_{T^{\mathrm{fin}}/2 \leqslant t < T^{\mathrm{fin}}} \| w_t^{\mL}
         \|_{H^{\ve}} \leqslant C (T^{\mathrm{fin}},
         \mathbf{N}^{\kappa}_{T^{\mathrm{fin}}}) \;,
      \end{equ}
      and as an application of Proposition~\ref{prop:wp-phi} we would be able
      to extend the domain of definition of the mild solution.
   \end{proof}

   \begin{proof}[of Theorem~\ref{thm:existence-uniqueness-global-solutions}]
      Suppose by contradiction that $ T^{\mathrm{fin}} < \infty $ so that, by
      Corollary~\ref{cor:explosion-condition}, $ T_{i} <
      T^{\mathrm{fin}} $ for every $ i \in \NN $. Since on the other hand,
      Lemma~\ref{lem:interval} implies that for $ \kappa >0 $
      sufficiently small
      \begin{equ}
         \sum_{i \in \NN } \big(T_{i + 1} - T_{i}\big) \geqslant \sum_{i \in
         \NN}\frac{1}{C (\mathbf{N}^{\kappa}_{T^{\mathrm{fin}}})(1 + \log{(1 +i)}
         )} \cdot \log{ \left( \frac{i^{2} + 2 i -
               C(\mathbf{N}^{\kappa}_{T^{\mathrm{fin}}})}{ i^{2}+
         C(\mathbf{N}^{\kappa}_{T^{\mathrm{fin}}})}  \right) } = \infty \;,
      \end{equ}
      our initial assumption must be false.
   \end{proof}

   \section{Global high-low weak solutions}
   In this section we prove Theorem~\ref{thm:weak}, regarding existence and
   uniqueness of  global
   weak solutions to \eqref{eqn:daPrato-Debussche} with initial datum $ u_{0} $ in $
   L^{2} $. 
   We start by introducing a suitable concept of weak solution to
   \eqref{eqn:daPrato-Debussche}.
   \begin{definition}[HL Solutions]\label{def:weak}
     For $ u_{0} \in L^{2} $ with $ \div(u_{0}) = 0 $ we say that a divergence-free process  $ v $ in
      $C([0, \infty); \mS^{\prime}
      (\TT^{2}; \RR^{2})) $ is a global high-low weak solution to
      \eqref{eqn:daPrato-Debussche} (HL solution for short) with initial condition $
      u_{0} $ if the following are satisfied by $ w = v - Y $, with $ Y
      $ given by \eqref{eqn:Y}:
      \begin{enumerate}
	\item For any $ T > 0 $ there exists a $ \lambda_{T} > 0 $ such that
	  for any $ \lambda \geqslant \lambda_{T} $ and $ t \in [0, T] $ the solution $ w $ is of the form $
	  w_{t} = w^{\mL, \lambda}_{t} + w^{\mH, \lambda}_{t} $, with $
w^{\mL, \lambda} $ and $ w^{\mH, \lambda} $ satisfying for 
            all $ \kappa \in (0, 1) $ 
            \begin{equs}
	      w^{\mL, \lambda} & \in L^{2}([0, T]; H^{1}) \cap L^{\infty}([0,
	      T]; L^{2}) \;, \\ 
		w^{\mH, \lambda} & \in L^{2}([0, T]; \mC^{1 - 2\kappa}_{4}) \cap
                L^{\infty} ([0, T]; L^{2})\;.
            \end{equs}
            In addition for $ Q_{t} $ as defined in \eqref{e:Q}, $
	    w^{\mH} $ is defined by 
	    \begin{equ}
	      w^{\mH, \lambda}_{t} = \div(w_{t} \para \mH_{\lambda} Q_{t})\;.
	    \end{equ}
         \item Equation \eqref{eqn:w} is satisfied by $ w $ in the weak sense.
            Namely, for any $ T > 0 $ and  $ \varphi \in C^{\infty}([0, T] \times
            \TT^{2}; \RR^{2}) $ satisfying $ \div(\varphi) = 0 $:
            \begin{equs}
               \langle w_{T}, \varphi_{T} \rangle &-  \langle w_{0},
               \varphi_{0} \rangle \\
               & \quad = \int_{0}^{T} \langle w,
               (\partial_{t} \varphi + \Delta \varphi)  \rangle + \langle 
	       \div(w^{\otimes 2} + D \sotimes w + Y^{\otimes
	       2} ), \varphi \rangle \ud s \;.
            \end{equs}
      \end{enumerate}
   \end{definition}
   Now we can establish existence of weak solutions.

   \begin{lemma}\label{lem:existence-weak}
      Let $ \mN^{\prime} $ be the null set of
      Lemma~\ref{lem:prob-space}. Then for any $ \omega \not\in \mN^{\prime}
      $ and $ u_{0} \in L^{2} $ with $
      \div(u_{0}) = 0 $ there exists an HL solution to
      \eqref{eqn:daPrato-Debussche} with initial condition $ u_{0} $.
   \end{lemma}
   
   \begin{proof}
      We construct a sequence of solutions to smooth
      approximations of \eqref{eqn:w} and prove a uniform energy estimate that
      guarantees compactness of the sequence. Let us define $
      u_{0}^{n} = \mL_{n} u_{0} $ and $ X^{n} = \mL_{n} X $ and $ Y^{n} $ the
      solution to \eqref{eqn:Y} with $ X $ replaced by $ X^{n} $. Then set $
      D^{n} = 2 (X^{n} + Y^{n}) $ and let $ w^{n} $ be the smooth
      solution to
      \begin{equ}[e:phi-n]
         \partial_{t} w^{n} = \Delta w^{n} +\mathbf{P} \div (
            (w^{n} )^{\otimes 2} + D \sotimes w^{n} + (Y^{n})^{\otimes 2})\;,
      \end{equ}
      with $ w^{n}_{0} = u_{0}^{n} $. Furthermore, we introduce the following
analogues of \( \mathbf{N}^{\kappa}_{t} \) and $
\mathbf{L}^{\kappa}_{t} $, cf.\ \eqref{e:def-N}:
\begin{equs}
  \mathbf{L}_{t}^{n, \kappa} & = 1+ \sup_{0 \leqslant s \leqslant t}
   \left\{ \| X_{s}^{n} \|_{\mC^{- \kappa}}+ \| Y_{s}^{n} \|_{\mC^{2 \kappa}}
\right\} \;,\\
   \mathbf{N}_{t}^{n, \kappa} & = \mathbf{L}_{t}^{n, \kappa} + \sup_{0 \leqslant s
   \leqslant t}   \sup_{i \in \NN} \Big\{\| (2 \nabla_{\mathrm{sym}
      } \mL_{\lambda^{i}}X_{s}^{n} ) \reso P^{\lambda^{i}, n}_{s} -
   \mf{r}_{\lambda^{i}}^{n} (s) \mathrm{Id} \|_{\mC^{- \kappa}} \Big\} \;,
\end{equs}
where we have defined for all $ \lambda \geqslant 1 $
   \begin{equs}
      P^{\lambda, n} (t, x) & = (-\Delta /2 + 1)^{-1} 2 \nabla_{\mathrm{sym}}
      \mL_{\lambda} X^{n} (t, x) \;, \\
	\mf{r}_{\lambda}^{n} (t) & = \frac{1}{4} \sum_{k \in \ZZ^{2}_{*}}
      \frac{\mf{l}(| k |/ \lambda) \mf{l}(| k |/ n)}{| k |^{2}/2 + 1}(1 - e^{- 2| k |^{2} t}) \;,
      \qquad \mf{r}_{\lambda}^{n} (t) \leqslant c \log{ (\lambda \wedge n)} \;.
   \end{equs}
With this definition we have that $ \lim_{n \to \infty}
   \mathbf{N}^{n , \kappa}_{t}  (\omega) = \mathbf{N}^{\kappa}_{t} (\omega) $
   for all $ \omega \not\in \mN^{\prime} $, the null set of
   Lemma~\ref{lem:prob-space}, so in particular
   \begin{equ}
      \overline{\mathbf{N}}_{t}^{\kappa} (\omega) \overset{\mathrm{def}}{=}
      \sup_{n \in \NN} \mathbf{N}_{t}^{n, \kappa} (\omega) < \infty \;, \qquad
      \forall \omega \not\in \mN^{\prime}\;.
   \end{equ}

      \textit{Step 1: A priori bound.} Our first objective is to show that for
      any $ T > 0 $ and $ \kappa > 0 $ sufficiently small there exists a $ C (T,
      \overline{\mathbf{N}}_{T}^{\kappa}) $ such that
      \begin{equ}[e:unif-approx]
         \sup_{n \in \NN} \bigg\{  \sup_{0 \leqslant t \leqslant T}  \|
            w_{t}^{n, \mL} \|^{2} + \int_{0}^{T} \|
         w_{t}^{n, \mL} \|_{H^{1}}^{2} \ud t  \bigg\}\leqslant C(T,
            \overline{\mathbf{N}}_{T}^{\kappa}) \;.
      \end{equ}
   To this aim, in analogy to Definition~\ref{def:stop-time-and-lambda} and
   \eqref{eqn:def-phi-high}, let us define $ T_{0}^{n} = 0 $ and
\begin{equ}
   T_{i+1}^{n} (\omega, u_{0}^{n}) = \inf \{ t \geqslant T_{i}^{n}  \; \colon
      \; \| w^{n}_{t} \| \geqslant i+1 \} \;,
\end{equ}
and set $\lambda_{t}^{n}
\eqdef (1 + \| w^{n}_{T_{i}}
\|)^{3}$ for all $T_{i} \leqslant t < T_{i+1}$. And finally,
for $ Q^{n} $ solving $ (\partial_{t} - \Delta) Q^{n} = 2 X^{n} $ with $
Q_{0}^{n} = 0$ and $ Q^{n, \mH}_{t} = \mH_{\lambda_{t}^{n}} Q^{n}_{t}$ we set
\begin{equ}
   w^{n, \mH} = \mathbf{P} \div ( w^{n} \para Q^{n, \mH})\;, \qquad
   w^{n, \mL} = w^{n} - w^{n , \mH} \;.
\end{equ}
Now we want can follow verbatim the proofs of Corollary~\ref{cor:unifrm-bd} and
Proposition~\ref{prop:nrg-estimate}, so that we obtain for any $ T, \kappa > 0 $
   \begin{equs}[e:wn-bd]
      \ &\sup_{T_{i}^{n} \leqslant t < T_{i+1}^{n}} \| w^{n, \mL}_{t} \|^{2} + 
      \int_{T_{i}^{n}}^{T_{i+1}^{n}} \|
      w^{n, \mL} \|_{H^{1}}^{2}\ud s \\
      & \leqslant \left( \| w^{n, \mL}_{T_{i}^{n}} \|^{2} +  C_{3} (
      \mathbf{N}^{n, \kappa}_{T_{i+1}^{n}}) \right) \cdot \exp \{
         (T_{i+1}^{n} - T_{i}^{n}) [C_{2} ( \mathbf{N}^{n, \kappa }_{T_{i
      +1}^{n}}) + C_{1} \log{(\lambda_{T_{i}^{n}} \wedge n )}] \}\;.
   \end{equs}

   In particular, since for fixed $ n \in \NN $ the solution $ w^{n} $ is
   smooth for all times, we can follow Lemma~\ref{lem:interval} to obtain that
   for some increasing $ C \colon \RR_{+} \to \RR_{+} $ and uniformly over $
   n, i \geqslant i_{0}(u_{0}) $ where $ i_{0} (u_{0})$ is as in \eqref{e:i0}:
   \begin{equ}
      T_{i+1}^{n} - T_{i}^{n} \geqslant \frac{1}{C
         ( \overline{\mathbf{N}}_{T_{i+1}^{n}})(1 + \log{(1 +i)}
      )} \cdot \log{ \left( \frac{i^{2} + 2 i -
      C( \overline{\mathbf{N}}_{T_{i+1}^{n}})}{ i^{2}+ C(
      \overline{\mathbf{N}}_{T_{i+1}^{n}})}  \right) } \;.
   \end{equ}
   Hence, from the divergence of the sequence $ \sum_{i \in \NN_{+}} (i
   \log{i})^{-1} = \infty $ we deduce that for every $ T>0, i \in \NN, i
   > i_{0} (u_{0}) $ there exists
   a time $ \mf{t}(i, \overline{\mathbf{N}}^{\kappa}_{T}) \in (0, T] $,
   satisfying \(\mf{t}(i) = T \) for all $ i $ sufficiently large, such that
   \begin{equ}[e:bd-T]
      \inf_{n \in \NN} T_{i}^{n} \geqslant \mf{t}(i,
      \overline{\mathbf{N}}^{\kappa}_{T})\;, \qquad \forall i> 
      i_{0} (u_{0}) \;.
   \end{equ}
   In addition we deduce that there exists a $ \overline{\lambda}_{T} > 0 $
   such that
   \begin{equ}[e:lbar]
     \lambda^{n}_{t} \leqslant \overline{\lambda}_{T} \;, \qquad \forall t \in
     [0, T] \;, n \in \NN  \;.
   \end{equ}
   Then from \eqref{e:wn-bd} and \eqref{e:bd-T} we can conclude that \eqref{e:unif-approx} holds
   true. 
   In fact, we can go one step further and use \eqref{e:lbar} to introduce the
   processes
   \begin{equ}
     w^{n, \mH, \lambda} = \div (w^{n} \para \mH_{\lambda} Q^{n})\;, \qquad
     w^{n, \mL, \lambda} = w^{n} - w^{n, \mH, \lambda} \;, \qquad \forall \lambda
     \geqslant \overline{\lambda}_{T} \;.
   \end{equ}
   Then following all the previous calculations we obtain that for any $ \lambda
   \geqslant \overline{\lambda}_{T} $
   \begin{equ}[e:unif-bar-L]
     \sup_{n \in \NN} \bigg\{  \sup_{0 \leqslant t \leqslant T}  \|
       w_{t}^{n, \mL, \lambda} \|^{2} + \int_{0}^{T} \|
     w_{t}^{n, \mL, \lambda} \|_{H^{1}}^{2}  \ud t  \bigg\}\leqslant C( \lambda, T,
     \overline{\mathbf{N}}^{\kappa}_{T}) \;.
   \end{equ}
   This leaves us roughly in the classical setting for solutions to the Navier--Stokes
   equations and we can follow, with a few modifications, \cite[Chapter 3, Theorem
   3.1]{temam2001navier}. 

   \textit{Step 2: More a-priori estimates.} Of course control on
   $ w^{n, \mL} $ alone is not sufficient, since we are interested in $
   w^{n} = w^{n, \mH} + w^{n, \mL} $, so let us now include the
   high-frequency term. We find by Lemma~\ref{lem:parap} that for any $
   \alpha < 1 - \kappa - 1/ \mf{a} $ and some $ C > 0 $
   \begin{equ}[e:wn-firstbd]
      \| w^{n}_{t} \|_{H^{\alpha}}  \leqslant \| w^{n, \mH}_{t}
      \|_{H^{\alpha}} + \| w^{n, \mL}_{t} \|_{H^{\alpha} }
       \leqslant C \mathbf{N}_{t}^{\kappa}  +  \| w^{n, \mL}_{t}
      \|_{H^{\alpha}} \;,
   \end{equ}
   where we have used Lemma~\ref{lem:high-feq-ref}. In particular, we obtain
   by \eqref{e:unif-approx}
   that for any $ \alpha < 1 - \kappa- 1 / \mf{a} $
   \begin{equ}[e:wn]
     \sup_{n \in \NN}\left\{  \sup_{0 \leqslant t \leqslant T} \|
       w^{n}_{t} \|^{2} + \int_{0}^{T} \| w^{n}_{t} \|_{H^{\alpha}}^{2} \ud t \right\}\leqslant C(T,
      \overline{\mathbf{N}}^{\kappa}_{T})\;.
   \end{equ}
   Via the bound $ \| w^{n}_{t} \|_{H^{1- \kappa}} \lesssim \|
w^{n}_{t} \| \| Q^{n, \mH}_{t} \|_{\mC^{2 - \kappa}} + \| w^{n, \mL}_{t} \|_{H^{1 - \kappa}} $ we
   can further improve our estimate to obtain by Sobolev embedding
   \begin{equ}[e:l4]
     \sup_{n \in \NN} \| w^{n} \|_{L^{2}([0, T]; \mC^{- \kappa}_{4}) \cap
     L^{2}([0, T]; H^{1- \kappa})} \leqslant C (T, \overline{\mathbf{N}}^{\kappa}_{T}) \;.
   \end{equ}
   Next, to obtain
   the high-low frequency decomposition for $ w $ we have to establish a bound
   on $ w^{n, \mH, \lambda} $. From \eqref{e:l4} we obtain, for
   every $ \lambda \geqslant 1 $
   \begin{equ}[e:unif-bar-H]
      \| w^{n, \mH, \lambda} \|_{L^{2}([0, T]; \mC^{1 - 2
     \kappa}_{4})} \leqslant C( T, \overline{\mathbf{N}}^{\kappa}_{T}) \;.
   \end{equ}
   Now we are ready to deduce the required convergence.

   \textit{Step 3: Convergence and conclusion.} In view of \eqref{e:wn} and  \eqref{e:l4} there
   exists a
   subsequence $ \{ n_{k} \}_{k \in \NN} $ we have for $ k \to \infty $
   and some $ w $:
   \begin{equ}[e:conv-2]
      w^{n_{k}} \overset{*}{\rightharpoonup} w \text{ in } L^{\infty}([0, T];
      L^{2}) \;, \qquad w^{n_{k}} \rightharpoonup w \text{ in } L^{2}([0, T];
      H^{1 - \kappa })\;.
   \end{equ}
   where the arrows $ \rightharpoonup $ and $ \overset{*}{\rightharpoonup} $
   indicate weak and weak-$ * $ convergence respectively.
   Now, weak convergence in $
   L^{2}(\TT^{2}) $ is not sufficient to deduce that the limiting process
   satisfies \eqref{eqn:w}. For this purpose we want to additionally establish
   the following  strong convergence, for any $ \beta < 1- \kappa $:
   \begin{equ}[e:strong]
      w^{n_{k}} \to w \;, \qquad \text {
      (strongly) in } L^{2}([0, T]; H^{\beta}) \;.
   \end{equ}
   To obtain this result, we would like to apply the Aubin--Lions lemma, so we bound via \eqref{e:phi-n} and for $
   \kappa > 0 $ sufficiently small 
   \begin{equs}
      \| \partial_{t} w^{n} \|_{H^{-2 - \kappa }} & \lesssim \| w^{n} \| + \| (w^{n}
      )^{\otimes 2} + D \sotimes w^{n} + (Y^{n})^{\otimes 2})\|_{H^{-1 - \kappa}}\\
      & \lesssim \| w^{n} \| + \| w^{n} \|^{2} + \| D \|_{\mC^{- \kappa} } \|
      w^{n} \|_{H^{1- \kappa}}+ (\overline{\mathbf{N}}_{T}^{\kappa})^{2} \\
      & \lesssim  \| w^{n} \|(1 + \| w^{n} \|) + \overline{\mathbf{N}}_{T}^{\kappa}  \|
      w^{n} \|_{H^{1- \kappa}}+ (\overline{\mathbf{N}}_{T}^{\kappa})^{2}\;,
   \end{equs}
   where we used the compact embedding $ (w^{n})^{\otimes 2} \in L^{1} \subseteq
   H^{- 1 - \kappa}  $ in dimension $ d =2 $ for the nonlinear term. We therefore conclude that
   \begin{equ}[e:dtwn]
      \sup_{n \in \NN} \| \partial_{t} w^{n} \|_{ L^{2} ([0, T]; H^{-2 -
      \kappa} )} \leqslant C(T, \overline{\mathbf{N}}_{T}^{\kappa}) \;,
   \end{equ}
   so that \eqref{e:strong} follows indeed from Aubin--Lions. 
   Now we can deduce that the limit $ w $ of the
   subsequence $ w^{n_{k}} $ is a weak solution to \eqref{eqn:w}. In fact by \eqref{e:strong}
   \begin{equ}
      (w^{n})^{\otimes 2} \to w^{\otimes 2}\;, \qquad  D^{n} \sotimes w^{n} \to D
      \sotimes w \;, \qquad (Y^{n})^{\otimes 2} \to Y^{\otimes 2}
   \end{equ}
   strongly in $ L^{2}([0, T]; L^{1}), L^{2}([0, T]; \mC^{- \kappa}_{2}) $ and $
   C([0, T]; \mC^{2 \kappa}) $ respectively. Finally, we have to establish the
   high-low frequency decomposition for $ w $. From the strong convergence in
   \eqref{e:strong} we obtain that as $ n \to \infty $
   \begin{equ}
      w^{n, \mH, \lambda} =  \div(w^{n} \para \mH_{\lambda}
      Q^{n}) \to \div(w \para \mH_{\lambda} Q) = w^{\mH, \lambda} \;.
   \end{equ}
   In addition \eqref{e:unif-bar-H} guarantees that $ w^{\mH, \lambda} $ has
   the required regularity $ L^{2}([0, T]; \mC^{1- 2 \kappa}_{4} ) $. That $
   w^{\mH, \lambda} $ lies in $ L^{\infty}([0, T]; L^{2}) $ follows from
   \eqref{e:wn-firstbd} with $ \alpha = 0 $.
   Similarly, from \eqref{e:unif-bar-L} we obtain that $ w^{\mL, \lambda} = w -
   w^{\mH, \lambda}$ lies in $ L^{\infty}([0, T]; L^{2}) \cap L^{2}([0, T];
   H^{1}) $, as required. This completes the proof.
  \end{proof}
   Next we show that HL solutions are unique.

   \begin{lemma}\label{lem:uniq-weak}
      Let $ \mN^{\prime} $ be the null set of
      Lemma~\ref{lem:prob-space}. Then for any $ \omega \not\in \mN^{\prime}
      $ and any initial condition $ u_{0} \in L^{2} $ with $ \div(u_{0})=0 $
      there exists at most one HL solution to
      \eqref{eqn:daPrato-Debussche} with initial condition $ u_{0} $ as in Definition~\ref{def:weak}.
   \end{lemma}
   \begin{proof}
      Consider two HL solutions $ v = w + Y, \overline{v} = \overline{w} + Y $
      to \eqref{eqn:daPrato-Debussche} and write
     $ z = w - \overline{w} $ and for any $ \lambda \geqslant
     \lambda_{T} $ define $ z^{\mL, \lambda} = w^{\mL, \lambda} -
     \overline{w}^{\mL, \lambda} , z^{\mH, \lambda}= z - z^{\mL, \lambda}$.
     Since in the first few steps we do not care about the choice of $ \lambda
     $, we omit it from the notation (meaning that we write $ z^{\mL},
     z^{\mH} $ in place of $ z^{\mL, \lambda} $ and $ z^{\mH, \lambda} $), up to the last step.
     We can compute via \eqref{e:wL}:
\minilab{e:fourTerms3}
     \begin{equs}
       \partial_{t} \frac{1}{2} \| z^{\mL} \|^{2} &=  \langle z^{\mL}, \Delta
	z^{\mL} +  \div ( 2 ( \mL_{\lambda} X ) \sotimes z^{\mL}) \rangle
	\label{e:z1} \\
	&\quad + \langle z^{\mL},  \div ( 2 (\mH_{\lambda} X) \sotimes z^{\mL} -
	2 (\mH_{\lambda} X ) \rpara z^{\mL} )\rangle \label{e:z2} \\
	&\quad + \langle z^{\mL},  \div ( 2 X \sotimes z^{\mH} - 2 (\mH_{\lambda}
	X ) \rpara z^{\mH}  + C^{\para}(z, \mH_{\lambda}
	Q))\rangle \label{e:z3} \\
	&\quad + \langle z^{\mL},   \div (w^{\otimes 2} - \overline{w}^{\otimes 2} +
	2 Y \sotimes z )\rangle \label{e:z4}\;.
      \end{equs}
      Let us note that the equality we have written has to be justified, but
      its proof follows from the regularity assumptions in
      Definition~\ref{def:weak} along the same estimate we will use below to
      obtain uniqueness. Also, recall that contrary to many previous calculations, $
      \lambda $ is a fixed and arbitrary large parameter. As usual we
      proceed one term at a time.

      \textit{Step 1: \eqref{e:z1}.} We estimate via
      Lemma~\ref{lem:low-freq}
      \begin{equs}
	\langle z^{\mL}, \Delta z^{\mL} +  \div ( 2 ( \mL_{\lambda} X ) \sotimes
	z^{\mL}) \rangle & \leqslant - \| z^{\mL} \|_{H^{1}} + C \|
	\mL_{\lambda} X \|_{\infty} \| z^{\mL} \| \| z^{\mL} \|_{H^{1}} \\
  & \leqslant - \| z^{\mL} \|_{H^{1}} + C \lambda^{2 \kappa}
  \mathbf{N}^{\kappa}_{T} \| z^{\mL} \| \| z^{\mL} \|_{H^{1}}\;.
  \label{e:step1}
      \end{equs}

      \textit{Step 2: \eqref{e:z2}.} Here we obtain
for any $ \delta \in (0, 1) $
      \begin{equs}
	\langle z^{\mL},  \div ( 2 & (\mH_{\lambda} X) \sotimes z^{\mL} -
	2 (\mH_{\lambda} X ) \rpara z^{\mL} )\rangle \\
	&=  \langle z^{\mL},  \div ( 2 (\mH_{\lambda} X) \reso z^{\mL} + 2
	(\mH_{\lambda} X ) \para z^{\mL} )\rangle \\
	&\leqslant C  \| z^{\mL} \|_{H^{1}} \| (\mH_{\lambda} X) \reso z^{\mL} +
	(\mH_{\lambda} X ) \para z^{\mL} \|\\
	&\leqslant C  \| z^{\mL} \|_{H^{1}} \mathbf{N}_{T}^{\kappa} \|
	z^{\mL} \|_{H^{2 \kappa}} \\
	&\leqslant \delta  \| z^{\mL} \|_{H^{1}}^{2} + C(\delta ,
  \mathbf{N}_{T}^{\kappa}) \| z^{\mL} \|^{2}\;, \label{e:step2}
      \end{equs}
      where the last inequality follows by interpolation $ \|
      z^{\mL} \|_{H^{2 \kappa}} \lesssim \| z^{\mL} \|^{1 - 2 \kappa} \|
      z^{\mL} \|^{2 \kappa}_{H^{1}} $ and by
      Young's inequality for products with $ p = 2/(1+ 2\kappa), q =
      2/(1 - 2 \kappa) $.

      \textit{Step 3: \eqref{e:z3}.} Here we estimate
      \begin{equs}
	\langle z^{\mL},   \div &  ( 2 X \sotimes z^{\mH} - 2 (\mH_{\lambda}
	X ) \rpara z^{\mH}  + C^{\para}(z, \mH_{\lambda}
	Q))\rangle \\
	& \lesssim  \| z^{\mL} \|_{H^{1}} (\| X \sotimes z^{\mH} -  (\mH_{\lambda}
	X ) \rpara z^{\mH} \| + \|C^{\para}(z, \mH_{\lambda} Q)\|) \\
	& \lesssim  \| z^{\mL} \|_{H^{1}}( \mathbf{N}^{\kappa}_{T} + \|
	  \mL_{\lambda} X\|_{\infty} )\| z^{\mH} \|_{H^{2 \kappa}} + \|
	C^{\para}(z, \mH_{\lambda}Q) \| ) \;.
      \end{equs}
      For the commutator $ C^{\para} $ we then proceed similarly to the proof
      of Lemma~\ref{lem:bd-4}, namely from \eqref{e:com}
      \begin{equ}
	\| C^{\para}(z, \mH_{\lambda} Q) \| \leqslant \| \mathrm{Tr}( \nabla z
	\para \nabla \mH_{\lambda}Q)\| +\| ( (\partial_{t} - \Delta) z) \para \mH_{\lambda}
	Q \|\;.
      \end{equ}
      For the first term we have, since $ Q \in \mC^{2 - \kappa} $:
      \begin{equ}
	\| \mathrm{Tr}( \nabla z \para \nabla \mH_{\lambda}Q)\| \lesssim \|
	\nabla z \|_{H^{-1 + 2 \kappa}} \| \nabla \mH Q \|_{\mC^{1 - \kappa}}
	\lesssim \mathbf{N}_{T}^{\kappa} \| z \|_{H^{2 \kappa}} \;.
      \end{equ}
      For the second term we have
      \begin{equ}
	\| ( (\partial_{t} - \Delta) z) \para \mH_{\lambda} Q \| \leqslant \|
	( \partial_{t} - \Delta) z \|_{H^{-2 + 2 \kappa}} \| Q
	\|_{\mC^{2 - \kappa}}\;,
      \end{equ}
      and from \eqref{eqn:w} we obtain that $ z $ solves
      \begin{equ}
	(\partial_{t} - \Delta)z = \mathbf{P} \div ( z \sotimes (w +
	\overline{w}) + D \sotimes z ) \;,
      \end{equ}
      so that
      \begin{equ}
	\| (\partial_{t} - \Delta)z \|_{H^{-2 + 2 \kappa}} \lesssim \| z
	\sotimes (w + \overline{w}) \|_{H^{-1 + 2 \kappa}} + \| D \sotimes z \|
	_{H^{-1 + 2 \kappa}}\;.
      \end{equ}
      Then, for the first quantity we have by Sobolev embedding that
      \begin{equ}
	\| z \sotimes (w + \overline{w}) \|_{H^{-1 + 2 \kappa}} \lesssim \|
	z \sotimes (w + \overline{w}) \|_{\mC_{1}^{3 \kappa}} \lesssim \| z
	\|_{H^{3 \kappa}} \| w + \overline{w} \|_{H^{3 \kappa}}\;,
      \end{equ}
      and for the second quantity
	$\| D \sotimes Z \|_{H^{-1 + 2 \kappa}} \lesssim \| D \|_{\mC^{- \kappa}}
	\| z \|_{H^{2 \kappa}}$.
	Hence in total for the commutator
      \begin{equ}
	\| C^{\para} (z, \mH_{\lambda}Q) \| \lesssim
	\mathbf{N}^{\kappa}_{T} \| z \|_{H^{3 \kappa} } (1 + \| w +
	\overline{w} \|_{H^{3 \kappa}}) \;.
      \end{equ}
      So overall, via Lemmma~\ref{lem:high-feq-ref} and by Young's inequality
      for products we can conclude that for any $ \delta \in (0, 1) $
      \begin{equs}
         \langle z^{\mL}&,   \div   ( 2 X \sotimes z^{\mH} - 2 (\mH_{\lambda} X
         ) \rpara z^{\mH}  + C^{\para}(z, \mH_{\lambda} Q))\rangle  \label{e:step3} \\
         & \leqslant C \| z^{\mL} \|_{H^{1}} \{ ( \| z^{\mH} \|_{H^{3 \kappa}} +
         \| z \|_{H^{3 \kappa} } )\mathbf{N}^{\kappa}_{T} (1 + \lambda^{ 2
         \kappa} + \| w + \overline{w} \|_{H^{3 \kappa}}) \} \\
         & \leqslant  \delta \| z^{\mL} \|^{2}_{H^{1}} + C(\delta,
         \mathbf{N}^{\kappa}_{T}, \lambda )  \{ ( \| z^{\mH} \|_{H^{3 \kappa}} +
         \| z \|_{H^{3 \kappa} } ) (1  + \| w +
      \overline{w} \|_{H^{3 \kappa}}) \}^{2}.
      \end{equs}

      \textit{Step 4: \eqref{e:z4}.} Here we estimate
      \begin{equ}
	\langle z^{\mL},   \div (w^{\otimes 2} - \overline{w}^{\otimes 2} +
	2 Y \sotimes z )\rangle \lesssim \| z^{\mL} \|_{H^{1}}( \| z \sotimes
	(w + \overline{w}) \| + \mathbf{N}_{T}^{\kappa} \| z \| )\;.
      \end{equ}
      Regarding the term involving $ z \sotimes (w + \overline{w}) $, we
      decompose it as
      \begin{equs}
	z \sotimes (w + \overline{w}) &=  z^{\mL} \sotimes (w^{\mL} +
	\overline{w}^{\mL}) + z^{\mL} \sotimes (w^{\mH} +
	\overline{w}^{\mH}) \\
	&\quad +z^{\mH} \sotimes (w^{\mL} +
	\overline{w}^{\mL}) +  z^{\mH} \sotimes (w^{\mH} +
	\overline{w}^{\mH})\;,
      \end{equs}
      in order to use the different regularity and integrability bounds on the
      high and low frequency terms.
      For the low frequency term we use
      Gagliardo--Nirenberg to bound
      \begin{equs}
         \| z^{\mL} \sotimes (w^{\mL} + \overline{w}^{\mL}) \| & \lesssim \|
         z^{\mL} \|_{L^{4}} \| w^{\mL} + \overline{w}^{\mL} \|_{L^{4}} \\
         & \lesssim
         \| z^{\mL} \|^{\frac{1}{2}} \| z^{\mL} \|^{\frac{1}{2}}_{H^{1}} \|
         w^{\mL} + \overline{w}^{\mL} \|^{\frac{1}{2}} \| w^{\mL} +
         \overline{w}^{\mL} \|^{\frac{1}{2}}_{H^{1}}\;.
      \end{equs}
      For the cross term we bound via the Riesz--Thorin interpolation theorem
      \begin{equs}
         \| z^{\mL} \sotimes (w^{\mH} + \overline{w}^{\mH}) \| & \lesssim \|
         z^{\mL}\|_{L^{4}} \| w^{\mH} + \overline{w}^{\mH} \|_{L^{4}} \\
         & \lesssim \| z^{\mL} \|^{\frac{1}{2}} \| z^{\mL}
         \|_{H^{1}}^{\frac{1}{2}} \| w + \overline{w} \|^{\frac{1}{2}} \| w +
         \overline{w} \|_{\infty}^{\frac{1}{2}} \;,
      \end{equs}
      and in addition by Besov embeddings we find that (provided $ \kappa $ is
      sufficiently small), as we are in dimension $ d=2 $ 
      \begin{equ}
         \| w + \overline{w} \|_{\infty} \lesssim \| w + \overline{w}
         \|_{\mC^{1 - 3 \kappa}_{4}} \;.
      \end{equ}
      We can use similar bounds on all the remaining terms to eventually obtain
      \begin{equ}
         \| z \sotimes (w + \overline{w}) \| \lesssim \| z
         \|^{\frac{1}{2}}  \opnorm{z}^{\frac{1}{2}}_{\lambda} (\| w \| + \| \overline{w}
         \|)^{\frac{1}{2}} (\opnorm{w}_{\lambda} + \opnorm{
         \overline{w}}_{\lambda} )^{\frac{1}{2}}
         \;,
      \end{equ}
      where for a function $ \varphi $, which for any $ \lambda \geqslant
\lambda_{T} $
      can be decomposed as $ \varphi =
      \varphi^{\mL, \lambda} + \varphi^{\mH, \lambda} $ we have defined
      \begin{equ}[e:nnorm]
         \opnorm{\varphi}_{\lambda} = \| \varphi^{\mL, \lambda} \|_{H^{1}} + \|
         \varphi^{\mH, \lambda } \|_{\mC^{1 - 3 \kappa}_{4}} \;.
      \end{equ}
      Then by Young's inequality for products with conjugate exponents $ p = 4/3, q = 4 $ we can
      bound for any $ \delta \in (0, 1) $ and a suitable $ C(\delta ) > 0 $:
      \begin{equs}
	\| z^{\mL}&  \|_{H^{1}}( \| z \sotimes
	(w + \overline{w}) \| + \mathbf{N}_{T}^{\kappa} \| z \| )\\
	&\leqslant \delta  
  \opnorm{z}^{2}_{\lambda} + C(\delta, \mathbf{N}_{T}^{\kappa}) \| z \|^{2}
  (1 + (\| w \|+ \|\overline{w} \|)^{2} (\opnorm{ w}_{\lambda} +
  \opnorm{\overline{w}}_{\lambda})^{2}) \;. \label{e:step4}
      \end{equs}

      \textit{Step 5: Conclusion.} Before we put together all our estimates,
      let us observe that in Steps $ 3 $ and $ 4 $ our bounds depend
      on the norms of $ z $ and $ z^{\mH, \lambda } $ and not just $ z^{\mL,
      \lambda} $. Here to
      obtain uniqueness we will make use of our freedom of choice for $ \lambda
      $. Let us start by considering the $ H^{\alpha} $ norm of $  z  $ for $
      0 \leqslant  \alpha < 1 - \kappa  $. By Lemma~\ref{lem:high-freq-reg-loss} we obtain:
      \begin{equs}
         \| z \|_{H^{\alpha}} & \leqslant \| z^{\mL, \lambda} \|_{H^{\alpha}} +
         \| z^{\mH, \lambda} \|_{H^{\alpha}} = \|
         z^{\mL, \lambda} \|_{H^{\alpha}} + \| \div ( z \para \mH_{\lambda} Q )
         \|_{H^{\alpha}} \\
         & \leqslant \| z^{\mL, \lambda} \|_{H^{\alpha}}  + C\| z
         \|_{H^{\alpha}} \| \mH_{\lambda} Q \|
         _{\mC^{\alpha + 1}} \\
         & \leqslant \| z^{\mL, \lambda} \|_{H^{\alpha}} + C
         \| z \|_{H^{\alpha}} \lambda^{-(1 - \kappa - \alpha)} \| Q \|_{\mC^{2 - \kappa}} \;.
      \end{equs}
      In particular, if we choose $ \overline{\lambda}(\alpha, \kappa, T)\geqslant 1 $ sufficiently large,
      so that
      \begin{equ}
         C \mathbf{N}^{\kappa}_{T} \leqslant \frac{1}{2} \{  \overline{\lambda}
         (\alpha, \kappa, T) \}^{1 - \kappa - \alpha}\;,
      \end{equ}
      we obtain for all $ \lambda \geqslant \overline{\lambda}(\alpha, \kappa,
      T) \vee \lambda_{T} $:
      \begin{equs}[e:z11]
         \| z \|_{H^{\alpha}} & \leqslant 2 \| z^{\mL, \lambda} \|_{H^{\alpha}}
         \;, \\
         \| z^{\mH, \lambda} \|_{H^{\alpha}}
         & \leqslant \| z \|_{H^{\alpha}}  + \| z^{\mL, \lambda } \|_{H^{\alpha}} \leqslant
         3 \| z^{\mL, \lambda} \|_{H^{\alpha}} \;.
      \end{equs} 
      Then choose $ \alpha = 1 - 2 \kappa $, so that from the
      Besov embedding $ H^{\alpha} \subseteq B^{\alpha -1}_{4, \infty} $,  for all $
      \lambda \geqslant \overline{\lambda} (1 -2 \kappa, \kappa, T) \vee
\lambda_{T}   $ we obtain
      \begin{equ}[e:z22]
         \| z^{\mH, \lambda} \|_{\mC^{1 - 3 \kappa}_{4}} \lesssim \| z^{\mL,
         \lambda } \|_{H^{1 - 2 \kappa}} \;.
      \end{equ}
      We deduce that for $ \opnorm{\cdot}_{\lambda} $ defined by
      \eqref{e:nnorm} we have $ \opnorm{z}_{\lambda} \lesssim \| z^{\mL, \lambda}
      \|_{H^{1}} $.

      We are now ready to collect the bounds from the
      previous steps: \eqref{e:step1}, \eqref{e:step2}, \eqref{e:step3} and
      \eqref{e:step4} (for sufficiently small $ \delta \in (0, 1) $), in
      combination with \eqref{e:z11} and \eqref{e:z22}. We find
      that for any $ \lambda \geqslant \overline{\lambda} (1 - 2\kappa, \kappa,
      T ) \vee \lambda_{T} $, choosing $ \delta > 0 $ sufficiently small in the bounds above
      \begin{equs}
         \partial_{t} \frac{1}{2} \| z^{\mL, \lambda} \|^{2} \leqslant & -
         \frac{3}{4}  \| z \|_{H^{1}}^{2} \\
         & + C( \mathbf{N}^{\kappa}_{T}, \lambda) (\| z^{\mL, \lambda } \| + \|
         z^{\mL, \lambda} \|_{H^{3 \kappa}} )^{2} (1 + \|
         w \|_{H^{3 \kappa}} + \| \overline{w} \|_{H^{3 \kappa}})^{2} \\
         & + C ( \mathbf{N}^{\kappa}_{T}) \| z^{\mL, \lambda} \|^{2} (1 + (\| w \|+
         \|\overline{w} \|)^{2} (\opnorm{ w}_{\lambda} +
      \opnorm{\overline{w}}_{\lambda})^{2})\;.
      \end{equs}
      As for the $ H^{3 \kappa} $ norm of $ z^{\mL, \lambda} $, by interpolation $ \|
      z^{\mL, \lambda } \|_{H^{3 \kappa}} \lesssim \| z^{\mL, \lambda} \|^{1 - 3 \kappa} \|
      z^{\mL, \lambda} \|^{3 \kappa}_{H^{1} }$. Hence by Young's inequality for products
      with conjugate exponents $ p = \frac{2}{6 \kappa}, q = \frac{2}{2 - 6
      \kappa}$ (as usual this is well-defined only for $ \kappa $ small) we
obtain for any $ A > 0, \delta \in (0,1)  $
      \begin{equ}
         \| z^{\mL, \lambda} \|^{2}_{H^{3 \kappa}} A  \leqslant  \delta \| z^{\mL,
         \lambda} \|_{H^{1}}^{2} +
         C(\delta) A^{\frac{q}{2}} \| z \|^{2}
         = \delta \| z^{\mL, \lambda} \|_{H^{1}}^{2} + C(\delta) A^{\frac{q}{2}} \| z
         \|^{2}\;.
      \end{equ}
      Of course, we want to apply this inequality with $ A^{\frac{1}{2}
      } = 1 + \| w \|_{H^{3 \kappa}} + \| \overline{w} \|_{H^{3 \kappa}} $, and
      for the last two terms we can apply the same line of inequalities to
      obtain, for $ \vt(\kappa) = \frac{4 \kappa}{1 - 4\kappa} $ 
      \begin{equ}
         \| w \|^{q}_{H^{3 \kappa}} \lesssim \| w \|^{q (1 - \vt(\kappa) )} \|
         w \|^{q \vt(\kappa)}_{H^{1 - 3\kappa}} = \| w \|^{\frac{1}{(1 -
         3\kappa)^{2}}}
         \| w \|^{\frac{3 \kappa}{(1 - 3 \kappa)^{2}}}_{H^{1 - 3\kappa} } \;.
      \end{equ}
      Hence, since for $ \kappa $ small $ \frac{3 \kappa}{(1 - 3
      \kappa)^{2}} \leqslant 2 $, we can further simplify our estimate to
      obtain
      \begin{equ}
         \partial_{t} \frac{1}{2} \| z^{\mL, \lambda} \|^{2} \leqslant  -
         \frac{1}{2} \| z \|_{H^{1}}^{2}
         + C( \mathbf{N}^{\kappa}_{T}, \lambda, M_{T}) \| z^{\mL, \lambda}
          \|^{2}  (1 + (\opnorm{ w}_{\lambda} +
          \opnorm{\overline{w}}_{\lambda})^{2})\;,
      \end{equ}
      where we have additionally defined $ M_{T} = \| w
      \|_{L^{\infty}([0, T]; L^{2})} + \| \overline{w}
      \|_{L^{\infty}([0, T]; L^{2})}$, which is finite from the definition of
      HL solutions. We can now deduce that for any $ \mf{t} \in (0, T] $
      \begin{equs}[e:finalZ]
         \sup_{0 \leqslant s \leqslant \mf{t}} & \| z^{\mL, \lambda}_{s} \|^{2} \\
         & \leqslant
         \Big( \sup_{0 \leqslant s \leqslant \mf{t}} \| z^{\mL, \lambda}_{s}
         \|^{2} \Big) C( \mathbf{N}_{T}^{\kappa}, \lambda, M_{T})
         \int_{0}^{\mf{t}} 1 + \opnorm{w_{s}}^{2}_{\lambda} + \opnorm{ \overline{w}_{s}
         }^{2}_{\lambda}\ud s \;.
      \end{equs}
      From the regularity assumptions on HL solutions in
      Definition~\ref{def:weak} we know that $ \int_{0}^{T}1 +
      \opnorm{w_{s}}^{2}_{\lambda} + \opnorm{w_{s}}^{2}_{\lambda} \ud s < \infty $. In particular
      by dominated convergence
      \begin{equ}
         \, [0, T] \ni \mf{t} \mapsto g( \mf{t} ) \eqdef
         \int_{0}^{\mf{t}} 1 + \opnorm{w_{s}}^{2}_{\lambda} + \opnorm{
         \overline{w}_{s}}^{2}_{\lambda}\ud s
      \end{equ}
      is a continuous map (and hence equicontinuous on the compact interval $
      [0, T] $).
      From \eqref{e:finalZ} we can conclude, by a contraction argument, that $ z = 0 $ on $
      [0, \mf{t}] $, for $ \mf{t} $ sufficiently small depending only on the
      modulus of continuity of $ g $. Hence we obtain also $ z_{\mf{t}} = 0 $
      for all $ \mf{t} > 0$ by iterating the argument and the proof is complete.
   \end{proof}

\section{The symmetrised operator}\label{sec:sym-op}
This section is devoted to the construction of the time-dependent operator $ t
\mapsto \mA_{t}  $ as in \eqref{eqn:def-A} and its approximations $
(\mA^{\lambda}_{t})_{\lambda \geqslant 1} $ as in \eqref{eqn:def-A-l}. The
construction is overall analogous to the construction of the 2D Anderson
Hamiltonian by Allez and Chouk \cite{AllezChouk}, although presently we are treating a vector-valued and
time-dependent case. 
The fundamental step is the construction of a
continuous map between the space $ \Xi $ of enhanced noises and the space $
\mathbf{C}_{\mathrm{op}} $ of closed self-adjoint operators with the graph
distance \cite[IV.2.4]{Kato1995} (convergence in this distance is implied by
convergence in the resolvent sense, which is the only one we will use here). We
define the space of \textit{enhanced noises}
   $\Xi_{\kappa} \subseteq  \mC^{-1 - \kappa}(\TT^{2}; \mathbf{M}^{2}) \times
   \mC^{ - \kappa}(\TT^{2}; \mathbf{M}^{2})$ by
          \begin{equ}[e:modelled]
             \Xi_{\kappa} = \overline{\{ ( \mathbf{X}_{1},
                      \mathbf{X}_{1}  \reso
                   (- \Delta /2 + 1)^{-1} \mathbf{X}_{1}- c)  \;
                   \colon \; \mathbf{X}_{ 1}  \in \mS (\TT^{2} ;
             \mathbf{M}^{2}) \;, \  c \in \RR  \}} \;,
         \end{equ}
   where the closure is taken with respect to the $ \mC^{-1 - \kappa}(\TT^{2}; \mathbf{M}^{2}) \times
   \mC^{ - \kappa}(\TT^{2}; \mathbf{M}^{2})$ product norm. We refer to these as
\emph{enhanced} noises because our purpose is to define the operator
\begin{equ}
\mf{A} = \frac{1}{2} \Delta + \mathbf{X}_{1} \;,
\end{equ}
but if $ \mathbf{X}_{1} \in \mC^{-1 - \kappa} $ for some $ \kappa > 0 $, there
is no canonical definition of such an operator and some additional information
(in terms of functionals of $ \mathbf{X}_{1} $) is required. This is because, for
generic $ \mathbf{X}_{1} \in \mC^{-1 - \kappa} $, the product $\mathbf{X}_{1}  \reso
                   (- \Delta /2 + 1)^{-1} \mathbf{X}_{1} $ is not well defined,
cf.\ Lemma~\ref{lem:parap}.
Eventually, we will associate to each element in the space $ \mX_{\kappa} $ a
closed operator, which will have as domain the space of
so-called strongly paracontrolled functions, which embeds into the following
space (with slightly simpler structure), for any $ \mathbf{X} =
(\mathbf{X}_{1}, \mathbf{X}_{2}) \in \Xi_{\kappa}$, for some \(\kappa>0\):
\begin{equs}
   \mX_{\kappa} (\mathbf{X}) & = \{ \varphi \in L^{2}  \; \colon \; \varphi
      = \varphi \para P + \varphi^{\sharp} \;, \quad \varphi \in
   H^{1 - \kappa} \;, \quad \varphi^{\sharp} \in H^{2 - 2\kappa}  \}\;, \\
   P & = (-\Delta/2 + 1)^{-1} \mathbf{X}_{1} \;, \label{e:P}
\end{equs}
with the associated norm
\begin{equ}
   \| \varphi \|_{\mX_{\kappa}} = \| \varphi \|_{H^{1 - \kappa}} + \| \varphi -
   \varphi \para P \|_{H^{2 - 2 \kappa}} \;.
\end{equ}
Then let us recall the following result concerning singular Hamiltonians.

\begin{proposition}[Allez--Chouk \cite{AllezChouk}]\label{prop:def-A}
   There exists a $ \kappa_{0} > 0 $ and a unique map  $ \mf{A} \colon \Xi  \to
   \mathbf{C}_{\mathrm{op}} $, where
   \begin{equ}
      \Xi = \bigcup _{0 < \kappa < \kappa_{0}} \Xi_{\kappa} \;,
   \end{equ}
   such that the following two properties are
   satisfied:
   \begin{enumerate}
      \item For any smooth $ \mathbf{X}= (\mathbf{X}_{1},
         \mathbf{X}_{2}) \in \mS (\TT^{2}; \mathbf{M}^{2}) \times \mS
	 (\TT^{2}; \mathbf{M}^{2}) \subseteq \Xi $ and $ \varphi \in
	 H^{2} $  we have
         \begin{equ}
            \mf{A}(\mathbf{X}) \varphi =  \frac{1}{2} \Delta \varphi + \mathbf{X}_{1} \para
            \varphi + \mathbf{X}_{1} \rpara \varphi + \varphi^{\sharp} \reso
            \mathbf{X}_{1} + \varphi \para \mathbf{X}_{2} + C^{\reso} (\varphi, P ,
            \mathbf{X}_{1})\;,
         \end{equ}
	 with $ P $ as in \eqref{e:P} and the commutator
         \begin{equ}
            C^{\reso}(\varphi, P, \mathbf{X}_{1}) = \mathbf{X}_{1} \reso
            (\varphi \para P) - \varphi \para ( P \reso \mathbf{X}_{1}) \;.
         \end{equ}
         In particular, if $
         \mathbf{X}_{2} = P \reso \mathbf{X}_{1} $ we recover $
         \mf{A} (\mathbf{X}) \varphi = \frac{1}{2} \Delta \varphi + \mathbf{X}_{1} \varphi $.
      \item For any sequence $ \{\mathbf{X}^{n}\}_{n \in \NN} \subseteq  \mS
   (\TT^{2}; \mathbf{M}^{2}) \times \mS (\TT^{2}; \mathbf{M}^{2}) $ such that
   for some $ \kappa < \kappa_{0} $ and $ \mathbf{X} \in \Xi_{\kappa} $, 
   $ \mathbf{X}^{n} \to \mathbf{X} $ in $ \Xi_{\kappa} $ as $ n \to \infty $, we have
   that \(\mf{A}(\mathbf{X}^{n}) \) converges in resolvent sense to $
   \mf{A}(\mathbf{X}) $.
   \end{enumerate}
   In addition, for any $ \kappa < \kappa_{0} $, there exist two continuous maps $ \mathbf{m},
   \mathbf{c} \colon \Xi_{\kappa} \to \RR_{+} $ (depending on $ \kappa $) such that $ [\mathbf{m}(\mathbf{X}), \infty) \subseteq
   \varrho(\mf{A}(\mathbf{X}))$, for any $ \mathbf{X} \in \Xi_{\kappa} $, where
   $ \varrho( \cdot) $ indicates the resolvent
   set of an operator, with the bound
   \begin{equ}
      \| (- \mf{A}(\mathbf{X}) + m)^{-1} \varphi \|_{\mX_{\kappa}}
      \leqslant \mathbf{c}(\mathbf{X}) \| \varphi \|_{L^{2}}  \;, \qquad \forall m
      \geqslant \mathbf{m}(\mathbf{X}) \;.
   \end{equ}
\end{proposition}
For a proof we refer for example to \cite[Proposition 4.13]{AllezChouk}: the
result is for the scalar setting, but its extension to the vector-valued
case is immediate.
Next we collect the Gaussian computations that are required for the
construction of the symmetrised operator $ \mA $. We start by
rewriting the driving noise in Fourier coordinates. In particular, we are
interested in the projection \(\mathbf{P} \Pi_{\times} \xi\) on divergence-free functions,
which can formally be represented in Fourier coordinates as follows:
\begin{align*}
   \mathbf{P} \Pi_{\times} \xi(t,x) = \sum_{k \in \ZZ^{2}_{*}} e^{2 \pi \iota k \cdot x} 
   \frac{( \partial_{t} \beta^{k, 1}_{t}  k_{2} - \partial_{t} \beta^{k,
   2}_{t} k_{1}) }{|k^{\perp}|} \frac{k^{\perp}}{|k^{\perp}|}\;,
\end{align*}
where \(\{ \beta_{t}^{k, i} \}_{i=1,2, k \in \ZZ^{2}_{*}}\) is a sequence of
complex Brownian motions with covariance structure
\[ \EE [ \partial_{t} \beta^{k , i}_{t} \partial_{t} \beta^{k^{\prime} ,
   j}_{s}] = \delta(t-s) 1_{\{i= j\}} 1_{\{k= -k^{\prime}\}}\;.\] 
   For our purposes it will be more convenient to set $ \zeta_{s}^{k} =
   (\beta^{k, 1}_{s} k_{1} - \beta^{k, 2}_{s} k_{2}) / | k | $, which is again
   a sequence of two-sided complex Brownian motions, with covariance structure
   \begin{equ}
      \EE [ \partial_{t} \zeta^{k}_{t} \partial_{t}
      \zeta^{k^{\prime}}_{s}] = \delta (t-s) 1_{\{ k = k^{\prime} \}} \;.
   \end{equ}
   With this notation, setting $ \{ e_{k} \}_{k \in \ZZ^{d}_{*}} $ a basis for
   the space of divergence-free functions, we can represent
\begin{equation*} 
   \mathbf{P} \xi(t, x) = \sum_{k \in \ZZ^{2}_{*}} \partial_{t}
   \zeta^{k}_{t}e_{k}(x) \;, \qquad e_{k}(x) = e^{2 \pi \iota k \cdot x }
   \frac{k^{\perp}}{|k^{\perp}|} \;.
\end{equation*}
In this context we can write $  \mL_{\lambda} X $ and $
(-\Delta / 2+1)^{-1} \mL_{\lambda} X$ in Fourier coordinates as
follows:
\begin{equs}[eqn:Fourier-X]
   \mL_{\lambda} X (t, x) & = \sum_{k \in \ZZ^{2}_{*}} e_{k}(x)
   F_{t}^{\lambda}(k; i, j) \;, \\
   ( -\Delta /2 + 1)^{-1} \mL_{\lambda} X (t, x) & = \sum_{k \in \ZZ^{2}_{*}}
   e_{k}(x) F_{t}^{\lambda} (k; i, j) (| k |^{2}+1)^{-1}  \;,
\end{equs}
with $ F $ being the stochastic integral
\begin{equ}
   F_{t}^{\lambda}(k) = \int_{0}^{t} e^{- | k |^{2}(t-s)} \ud
   \zeta^{k}_{s} \;.
\end{equ}
We are now ready to state the main result of this subsection, namely the
convergence of the stochastic terms required to make sense of the operator $
\mA $.

\begin{lemma}\label{lem:stochastic-bds}
   For any $ \kappa > 0 $, let $ \Xi_{\kappa} \subseteq  \mC^{-1 - \kappa}(\TT^{2}; \mathbf{M}^{2}) \times
   \mC^{ - \kappa}(\TT^{2}; \mathbf{M}^{2})$ be as in \eqref{e:modelled}.
   Furthermore, define for any $ \lambda > 1 $
   \begin{equ}[e:p-lambda]
      P^{\lambda} (t, x) = (-\Delta /2 + 1)^{-1} 2 \nabla_{\mathrm{sym}}
      \mL_{\lambda} X (t, x) \;.
   \end{equ}
   Then for any $ t \geqslant 0 $ there exists a
   distribution $  \nabla_{\mathrm{sym}}X_{t} \diamond
   P_{t}  $ in $  \mC^{- \kappa} (\TT^{2}; \mathbf{M}^{2 }) $, for which the
   following convergence holds, for any $ \kappa > 0 $, both in $ L^{p} ( \Omega ;  C_{\mathrm{loc}}(\RR_{+};
   \Xi_{\kappa}) )$ for any $ p \in [1, \infty) $ and almost surely:
   \begin{equ}
      ( 2 \nabla_{\mathrm{sym}}\mL_{\lambda^{n}} X \;, ( 2\nabla_{\mathrm{sym}}
      \mL_{\lambda^{n}} X) \reso P^{\lambda^{n}} - \mf{r}_{\lambda^{n}} \mathrm{Id} ) \to
      ( 2\nabla_{\mathrm{sym}}X, 2 \nabla_{\mathrm{sym}}X \diamond P) \;, 
   \end{equ}
   as \( n \to \infty \)\footnote{Here we view all random variables as
   time-dependent, so that the map $ \RR_{+} \ni t \mapsto
   (2 \nabla_{\mathrm{sym}}X_{t}, 2 \nabla_{\mathrm{sym}}X_{t} \diamond P_{t}) $ is
a continuous path with values in $ \Xi_{\kappa}$.}.
In addition, the renormalisation constants $ \mf{r}_{\lambda}(t) $ satisfy, for
some \( c > 0 \) and uniformly over all \( \lambda \geqslant 1 \) and $ t
\geqslant 0 $
   \begin{equ}[eqn:ren-const]
      \mf{r}_{\lambda} (t) = \frac{1}{4} \sum_{k \in \ZZ^{2}_{*}}
      \frac{\mf{l}(| k |/ \lambda)}{| k |^{2}/2 + 1}(1 - e^{- 2| k |^{2} t}) \;,
      \qquad \mf{r}_{\lambda} (t) \leqslant c \log{ \lambda} \;, 
   \end{equ}
   with $ \mf{l} $ as in Definition~\ref{def:cut-off}.
\end{lemma}

\begin{proof}
   We restrict ourselves to proving the convergence of the product $
   (\mL_{\lambda^{n}} \nabla_{\mathrm{sym}} X )\reso P^{\lambda^{n}} -
   \mf{r}_{\lambda^{n}} (t)$, as the
   convergence of $ \mL_{\lambda^{n}} \nabla_{\mathrm{sym}} X $ follows along
   similar calculations. We observe
   that both $ X $ and \( P^{\lambda} \) are
   Gaussian fields, so their product lives in the second and the zeroth chaos.
   We treat the two terms differently, since the renormalisation constant \(
   \mf{r}_{\lambda} \) is chosen exactly to cancel out the zeroth chaos.
   In components, the problem amounts to studying the following product:
   \begin{equ}
      ( 2\nabla_{\mathrm{sym}} \mL_{\lambda} X \reso P^{\lambda} )_{i, j} =
      \sum_{l= 1, 2} (\partial_{i} X^{l}_{\lambda} + \partial_{l}
      X^{i}_{\lambda}) \reso (-\Delta /2+1)^{-1} (\partial_{l} X^{j}_{\lambda} + \partial_{j}
      X^{l}_{\lambda}) \;,
   \end{equ}
   where $ X_{\lambda} = \mL_{\lambda} X = (X^{i}_{\lambda})_{i = 1, 2} $.
   Hence we are lead to consider a product of the
   following form, for $ i, j, l, m \in \{ 1, 2 \} $ 
   \begin{equ}
      \partial_{i} X^{j}_{\lambda} \reso (-\Delta /2 +1)^{-1} \partial_{l} X^{m}_{\lambda} \;.
   \end{equ}
   Using \eqref{eqn:Fourier-X} we can represent this product as follows:
   \begin{equs}
      \sum_{k, k^{\prime} \in \ZZ^{2}_{*}} \sum_{| c-d | \leqslant 1} e^{\iota k
      \cdot x} \varrho_{c} (k - k^{\prime} ) & \varrho_{d} (k^{\prime} )
      \mf{l} (| k - k^{\prime} | / \lambda ) \mf{l}(| k^{\prime} | / \lambda)
      \\
      & \cdot F_{t}(k - k^{\prime} ) F_{t}(k^{\prime} ) (| k^{\prime}  |^{2}/2+1)^{-1} 
      \mf{c}^{i, l}_{j , m} (k - k^{\prime} , k^{\prime}  ) \;.
   \end{equs}
   Here the coefficient $ \mf{c}^{i, l}_{j, m} $ is defined as
   \begin{equ}
      \mf{c}^{i, l}_{j, m} (k, k^{\prime})  = - k_{i} k^{\perp}_{j}
      k_{l}^{\prime} (k^{\prime} )^{\perp}_{m} | k |^{-1} | k^{\prime} |^{-1}
      \;.
   \end{equ}

   \textit{Zeroth chaos.} In particular, the zeroth chaos (the average) is
   given by the contraction along the line
   \( k - k^{\prime} = - k^{\prime} \), so that we find
   \begin{equs}
      \EE [\partial_{i} X^{j}_{\lambda} \reso (-\Delta/2+1)^{-1} \partial_{l}
      X^{m}_{\lambda}](x)& = \sum_{k \in \ZZ^{2}_{*}} \frac{\mf{l}^{2}(| k
      |/ \lambda)}{4 | k |^{4}(| k |^{2}/2 +1)}  (1 - e^{- 2 | k
      |^{2}t})\mf{c}^{i, l}_{j, m} (-k, k) \\
      &= \sum_{k \in \ZZ^{2}_{*}} \mf{l}^{2}(| k |/ \lambda)(1 - e^{- 2 | k
      |^{2}t})\frac{ k_{i} k^{\perp}_{j}
      k_{l}k^{\perp}_{m}}{4 | k |^{4}(| k |^{2}/2 + 1)}  \;,
   \end{equs}
   since $ \int_{0}^{t} e^{- 2 | k |^{2} (t -s)} \ud s =
   \frac{1}{2 | k |^{2}} (1 - e^{- 2 | k |^{2}t})$. In particular, it follows that the average is
   nonzero only in two cases, namely if either $ i = j = l=m $ or exactly two
   of the indices are $ 1 $ and the other two are $ 2 $ (in all other cases the
   sum is anti-symmetric). As a consequence of this observation we immediately
   obtain that
   \begin{equ}
      \EE [ ( 2 \nabla_{\mathrm{sym}} \mL_{\lambda} X \reso P^{\lambda} )_{i, j}]
      = 0 \;, \text{ if } i \neq j \;.
   \end{equ}
   We are thus left with computing the average $ \EE [(2\nabla_{\mathrm{sym}}
   \mL_{\lambda} X \reso P^{\lambda} )_{i, i}] $. This amounts to considering
   four different terms. We start by observing that
   \begin{equs}
      \EE \bigg[ \sum_{l = 1, 2} \partial_{i} X^{l}_{\lambda} & \reso (-\Delta
         /2+1)^{-1}
      \partial_{l} X^{i}_{\lambda} \bigg] \\
      & =  \sum_{k \in \ZZ^{2}_{*}} \sum_{l = 1,2}\mf{l}^{2}(| k |/ \lambda)\frac{ k_{i} k^{\perp}_{l}
      k_{l}k^{\perp}_{j}}{4 | k |^{4}( | k |^{2}/2 + 1)} (1 - e^{- 2 | k
      |^{2}t})\\
      & =  \sum_{k \in \ZZ^{2}_{*}} \mf{l}^{2}(| k |/ \lambda)\frac{ k_{i} \langle k^{\perp}, 
      k \rangle k^{\perp}_{j}}{4 | k |^{4}( | k |^{2}/2 + 1)} (1 - e^{- 2 | k
      |^{2}t})= 0 \;.
   \end{equs}
   Similarly also
   \begin{equ}
      \EE \left[ \sum_{l=1,2}  \partial_{l} X^{i}_{\lambda} \reso (-\Delta/2+1)^{-1}
      \partial_{i} X^{l}_{\lambda} \right] = 0 \;.
   \end{equ}
   In particular, we have reduced ourselves to computing
   \begin{equs}
      \EE [( 2 \nabla_{\mathrm{sym}} & \mL_{\lambda} X \reso P^{\lambda} )_{i,
      i}] \\
      & =  \sum_{l=1, 2} \EE \left[ \partial_{i} X^{l}_{\lambda} \reso
         (-\Delta/2+1)^{-1} (\partial_{i} X^{l}_{\lambda}) + \partial_{l} X^{i}_{\lambda} \reso
      (-\Delta/2+1)^{-1} \partial_{l} X^{i}_{\lambda} \right] \\ & =  \sum_{k \in
      \ZZ^{2}_{*}} \sum_{l =1,2} \frac{\mf{l}^{2}(| k |/ \lambda)}{4 | k |^{4}(
      | k |^{2}/2 +1) }  \left\{ k_{i}^{2} (k_{l}^{\perp})^{2} +
      k_{l}^{2}(k_{i}^{\perp})^{2} \right\} (1 - e^{- 2 | k
      |^{2}t})\\ 
      & = \sum_{k \in \ZZ^{2}_{*}} 
      \frac{\mf{l}^{2}(| k |/ \lambda)}{4 | k |^{2}(| k |^{2}/2 +1)}  \left\{ k_{i}^{2}  +
      (k_{i}^{\perp})^{2} \right\} (1 - e^{- 2 | k
      |^{2}t})\\
      & = \frac{1}{4} \sum_{k \in \ZZ^{2}_{*}} 
      \frac{\mf{l}^{2}(| k |/ \lambda)}{| k |^{2}/2 + 1} (1 - e^{- 2 | k
      |^{2}t})\;,
   \end{equs}
   which is the required quantity.

   \textit{Second chaos.} Instead, for the second chaos we can bound for any $
   p \geqslant 2 $ by Gaussian hypercontractivity
   \begin{equs}
      \EE \big[ \big\vert \Delta_{j} ( 2 \nabla_{\mathrm{sym}} \mL_{\lambda} X \reso P^{\lambda} )_{i,
      j} & - \mf{r}_{\lambda} (t) \big\vert^{p} (x) \big] \\
      & \lesssim \EE \left[
         \big\vert \Delta_{j} ( 2 \nabla_{\mathrm{sym}} \mL_{\lambda} X \reso
         P^{\lambda} )_{i, j} - \mf{r}_{\lambda} (t) \big\vert^{2} (x)
      \right]^{\frac{p}{2}} \;.
   \end{equs}
   Then, for the second moment we can estimate with $ \psi_{0}(k, k^{\prime}) =
   \sum_{| c -d | \leqslant 1} \varrho_{c} (k) \varrho_{d}(k^{\prime})$ 
   \begin{equs}
      \EE \big[
         \big\vert \Delta_{j} &( 2 \nabla_{\mathrm{sym}} \mL_{\lambda} X \reso
         P^{\lambda} )_{i, j}  - \mf{r}_{\lambda} (t) \big\vert^{2} (x)
      \big] \\
      &\lesssim \sum_{k, k^{\prime} \in \ZZ^{2}_{*} } \int_{0}^{t}
      \int_{0}^{t} \varrho_{j}^{2} (k+ k^{\prime}) \psi_{0}^{2} (k, k^{\prime})
      \mf{l}^{2} (| k | / \lambda ) \mf{l}^{2}(| k^{\prime} | / \lambda)
      \\
      & \qquad \qquad \qquad \qquad \cdot e^{- 2 (t -s) | k |^{2}} e^{- 2 (t -s^{\prime} ) | k^{\prime}
      |^{2}} (| k^{\prime}  |^{2}+1)^{-2} 
      |\mf{c}^{i, l}_{j , m} (k , k^{\prime}  ) |^{2} \ud s \ud s^{\prime} \\
      & \lesssim \sum_{k, k^{\prime} \in \ZZ^{2}_{*} }   \varrho_{j}^{2} (k
      ) \psi_{0}^{2} (k- k^{\prime} , k^{\prime})
      \mf{l}^{2} (| k - k^{\prime}  | / \lambda ) \mf{l}^{2}(| k^{\prime} | / \lambda)
      \frac{|\mf{c}^{i, l}_{j , m} (k - k^{\prime}  , k^{\prime}  ) |^{2}}{| k
      - k^{\prime} |^{2} | k^{\prime} |^{2} (| k^{\prime} |^{2}/2 + 1)^{2}} \\
      & \lesssim 2^{jd} \sum_{|k^{\prime}| \gtrsim 2^{j} }   
      \frac{|k^{\prime}  |^{4}}{|
      k^{\prime} |^{4} (| k^{\prime} |^{2} /2 + 1 )^{2}} \lesssim 2^{j d}
      \sum_{| k^{\prime} | \gtrsim 2^{j}} \frac{1}{| k^{\prime} |^{4}} \lesssim
      2^{j (d -2)} \lesssim 1 \;,
   \end{equs}
   since we are in dimension $ d =2 $. In this way we obtain for any $ \kappa >
   0 $ and $ p \geqslant 2 $ that
   \begin{equs}
      \sup_{\lambda \geqslant 1} \EE & \left[ \| (2 \nabla_{\mathrm{sym}}
      X^{\lambda}_{t}) \reso P^{\lambda}_{t} - \mf{r}_{\lambda}(t)
   \|_{\mB^{- \kappa}_{p,p}}^{p}\right] \\
   & \lesssim \sup_{\lambda \geqslant 1}
   \sup_{j \geqslant -1} \sup_{x \in \TT^{2}} \EE \left[ \big\vert \Delta_{j} (
         2 \nabla_{\mathrm{sym}} \mL_{\lambda} X \reso P^{\lambda} )_{i, j} -
      \mf{r}_{\lambda} (t) \big\vert^{2} (x)\right]^{\frac{p}{2}}   < \infty \;.
   \end{equs}
   From here to obtain convergence of the sequence in $ L^{p} $ for $ \lambda
   \to \infty $ follows along classical lines. Instead, let us address the
   almost sure convergence for the sequence $ \{ \lambda^{i} \}_{i \in \NN} $.
   To this aim, we have to bound for any $ i \in \NN $ the difference
   \begin{equ}
      \EE \left[ \| (2 \nabla_{\mathrm{sym}}
         X^{\lambda^{i}}_{t}) \reso P^{\lambda^{i}}_{t} -
         \mf{r}_{\lambda^{i}}(t) -  (2 \nabla_{\mathrm{sym}}
         X^{\lambda^{i+1}}_{t}) \reso P^{\lambda^{i+1}}_{t} +
         \mf{r}_{\lambda^{i+1}}(t)  \|_{\mB^{- \kappa}_{p,p}}^{p}\right]
   \end{equ}
   Following the previous calculation we are thus led to bound, for $ j
   \geqslant -1 $ and $ x \in \TT^{d} $
   \begin{equs}
      \EE & \left[ \left\vert \Delta_{j} \left[(2 \nabla_{\mathrm{sym}}
            X^{\lambda^{i}}_{t}) \reso P^{\lambda^{i}}_{t} -
         \mf{r}_{\lambda^{i}}(t) -  (2 \nabla_{\mathrm{sym}}
         X^{\lambda^{i+1}}_{t}) \reso P^{\lambda^{i+1}}_{t} +
      \mf{r}_{\lambda^{i+1}}(t)   \right]  (x) \right\vert^{2} \right] \\
      & \lesssim \sum_{k, k^{\prime} \in \ZZ^{2}_{*} } \int_{0}^{t}
      \int_{0}^{t} \varrho_{j}^{2} (k+ k^{\prime}) \psi_{0}^{2} (k, k^{\prime}) \\
      & \qquad \qquad \qquad \qquad \cdot \left\{ \mf{l} (| k | / \lambda^{i} ) \mf{l}(| k^{\prime} | /
         \lambda^{i}) -  \mf{l} (| k | / \lambda^{i+1}) \mf{l}(| k^{\prime} | /
      \lambda^{i+1})\right\}^{2} \\
      & \qquad \qquad \qquad \qquad \cdot e^{- 2 (t -s) | k |^{2}} e^{- 2 (t -s^{\prime} ) | k^{\prime}
      |^{2}} (| k^{\prime}  |^{2}+1)^{-2} 
      |\mf{c}^{i, l}_{j , m} (k , k^{\prime}  ) |^{2} \ud s \ud s^{\prime}\\
      & \lesssim \sum_{k, k^{\prime} \in \ZZ^{2}_{*}} \varrho_{j}^{2} (k+
      k^{\prime}) \psi_{0}^{2} (k, k^{\prime}) \frac{|\mf{c}^{i, l}_{j , m} (k
      , k^{\prime}  ) |^{2}}{| k |^{2} | k^{\prime} |^{2}( | k^{\prime} |^{2} +
      1)} \{ 1_{[\lambda^{i}, \lambda^{i+1}]} (k) + 1_{\{ [\lambda^{i},
      \lambda^{i+1}] \}}(k^{\prime}) \} \\
      & \lesssim  (\lambda^{i})^{- \frac{\kappa}{4}} \sum_{k, k^{\prime} \in
      \ZZ^{2}_{*}} \varrho_{j}^{2} (k+ k^{\prime}) \psi_{0}^{2} (k, k^{\prime})
      \frac{|\mf{c}^{i, l}_{j , m} (k
      , k^{\prime}  ) |^{2}}{| k |^{2 - \frac{\kappa}{4} } | k^{\prime} |^{2 -
         \frac{\kappa}{4} }( | k^{\prime} |^{2} + 1)} \\
         & \lesssim (\lambda^{i})^{- \frac{\kappa}{4}} 2^{\frac{\kappa}{2}
         j} \;,
   \end{equs}
   where in the last step we follow the previous calculations. We deduce that
   \begin{equs}
      \EE \bigg[ \| (2 \nabla_{\mathrm{sym}}
         X^{\lambda^{i}}_{t}) \reso P^{\lambda^{i}}_{t} -
         \mf{r}_{\lambda^{i}}(t) -  (2 \nabla_{\mathrm{sym}}
         X^{\lambda^{i+1}}_{t}) \reso P^{\lambda^{i+1}}_{t} + & 
         \mf{r}_{\lambda^{i+1}}(t)  \|_{\mB^{- \kappa}_{p,p}}^{p}\bigg] \\
         & \lesssim (\lambda^{i})^{- \frac{\kappa p }{2}} \;,
   \end{equs}
   so that the almost sure convergence follows, since by \eqref{e:def-lambda}
   we have $ \sum_{i \in \NN} (\lambda^{i})^{- \frac{\kappa p}{2}} < \infty$
   for \( p \geqslant 2 \) sufficiently large. The convergence uniformly in
   time follows by similar estimates, and this concludes the proof.
\end{proof}

\bibliographystyle{Martin}
\bibliography{bibliography}

\appendix

\section{Function spaces and paraproducts}\label{sec:function-spaces}
We define the space of Schwartz functions \(\mS(\TT^{2}; \RR^{d})= \bigcap_{k \in
\NN}C^{k}(\TT^{2}; \RR^{d})\) and their topological dual, the set of Schwartz
distributions \(\mS^{\prime}(\TT^{2}; \RR^{d})\). Then the Fourier transform
$\hat{\varphi}$ is defined for any distribution $\varphi \in \mS^{\prime}(\TT^2 ; \RR^d)$:
\[ \hat{\varphi} (k) =\mathcal{F} \varphi (k) = \int_{\TT^2} e^{- 2 \pi
   \iota k \cdot x} \varphi (x) \ud x, \qquad \hat{\varphi} : \ZZ^2
   \rightarrow \RR^d . \]
   We additionally define the space of mean-free Schwartz distributions $
   \mS^{\prime}_{\times} (\TT^{2}; \RR^{d}) = \{ \varphi \in  \mS^{\prime}
   (\TT^{2}; \RR^{d} )  \; \colon \; \hat{\varphi} (0) = 0  \}$.
   Then, for any $p\in [1,\infty], d\in\NN$ and \(O \subseteq \RR^{d}\) we denote with
$L^p(\TT^2;O)$ the
Banach space of measurable functions (modulo modifications on a null set) $\varphi \colon \TT^2 \to \RR^d$  such that the norm
\[ \| \varphi \|_{L^p(\TT^{2}; \RR^{d})} =\left( \int_{\TT^d} |\varphi (x)|^p \ud x\right)^{\frac{1}{p}}\]
is finite, with the usual convention for $ p = \infty $. For brevity we write
\begin{equ}
   \| \varphi \| = \| \varphi \|_{L^{2}} \;.
\end{equ}
Next we introduce the scale of mean-free Besov spaces
\(\mB^{\alpha}_{p,q}(\TT^{2} ; \RR^{d}) \subseteq \mS^{\prime}_{\times} (\TT^{2};
\RR^{d})\), for \(\alpha \in \RR\), \(p,q \in
[1, \infty]\). Having fixed a \(2-\)dimensional dyadic partition
of the unity \(\{\varrho_{j}\}_{j \geqslant -1}\) (see
\cite{BahouriCheminDanchin2011FourierAndNonLinPDEs}), the
spaces \(\mB^{\alpha}_{p,q}(\TT^{2}; \RR^{d})\) are defined via the norms:
\[ \| \varphi \|_{\mB^{\alpha}_{p,q}(\TT^{2}; \RR^{d})} = \bigg( \sum_{i
\geqslant -1} 2^{i \alpha q} \| \Delta_{i} \varphi \|_{L^{p}(\TT^{2}; \RR^{d})}^{q}
\bigg)^{\frac{1}{q}} \;,\]
with the Paley block $ \Delta_{i} \varphi $ defined in Section~\ref{sec:para}
below.
In particular we will distinguish the Bessel potential spaces, corresponding to $ p = q
= 2 $:
\[ H^{\alpha}(\TT^{2}; \RR^{d}) = \mB^{\alpha}_{2,2}(\TT^{2}; \RR^{d}) \;,\]
over which we will use the equivalent norm (recall that we are only considering
mean-free functions) 
\begin{equ}
   \| \varphi \|_{H^{\alpha}} = \| (- \Delta)^{\frac{\alpha}{2}} \varphi \| \;.
\end{equ}
Next we also distinguish the H\"older-Besov spaces
\[ \mC^{\alpha}(\TT^{2}; \RR^{d}) = \mB^{\alpha}_{\infty,
\infty}(\TT^{2}; \RR^{d}) \;, \qquad \mC^{\alpha}_{p} (\TT^{2}; \RR^{d}) =
\mB^{\alpha}_{p, \infty}(\TT^{2}; \RR^{d}) \;. \] 
For time-dependent functions we consider the space of Schwartz functions
\begin{equs}
   \mS(\RR \times \TT^{2}; O) = \Big\{ \varphi \colon &\RR \times  \TT^{2} \to O
 \ \colon \\
 &\sup_{t \in \RR, x \in \TT^{2}} \{ (1 + |t|)^{p} |
\partial^{\mu} \varphi |(t,x) \} < \infty, \ \forall p \geqslant 0, \mu \in
\NN^{3} \Big\}\;,
\end{equs}
and its topological dual \(\mS^{\prime}(\RR \times \TT^{2})\), the space of
Schwartz distributions.
For time-dependent measurable functions \(\varphi \colon [0,t] \to X\) for some
\(t>0\) and a Banach space \(X\) we introduce the spaces \(L^{p}_{t} X\),
for \(p \in [1, \infty]\) via the norm
\[ \| \varphi \|_{L^{p}_{t} X} = \bigg( \int_{0}^{t} \| \varphi (s)
\|_{X}^{p} \ud s \bigg)^{\frac{1}{p}}.\]

\subsection{Paraproducts}\label{sec:para}
Next consider, for \( \varphi \in \mS^{\prime} (\TT^{2}; \RR^{d}),\) the Paley
block
\[ \Delta_{i} \varphi(x) = \mF^{-1} \big( \varrho_{j}(\cdot) \mF \varphi
(\cdot) \big)(x) \;,\] 
as well as the paraproducts (the sum is only formal and its convergence has to be justified), for \( \varphi, \psi \in \mS^{\prime}
(\TT^{2}; \RR^{d}) \):
\begin{equs}
\varphi \para \psi (x) &= \sum_{ -1 \leqslant j \leqslant i -1}
\Delta_{j} \varphi (x) \sotimes \Delta_{i} \psi (x) \in \bM^{d} \;,\\
\varphi \reso \psi (x) & = \sum_{|i - j| \leqslant 1} \Delta_{i} \varphi (x)
\sotimes  \Delta_{j} \psi (x) \in \bM^{d} \;.
\end{equs}
So one can formally decompose the tensor product between two distributions
\( \varphi, \psi \) as:
\begin{equ}
\varphi \sotimes \psi  = \varphi \para \psi + \varphi\reso \psi +
\varphi \rpara \psi \;.
\end{equ}
In the hope that no confusion can occur, we will slightly abuse of the notation of
paraproducts, allowing it to denote both \emph{tensor products} as the one we
just described, and \emph{matrix multiplication} (which is just a contraction
along some index of the former). In particular, we will consider the situation
in which we are give a matrix $ M \in \mS^{\prime} (\TT^{2}; \mathbf{M}^{2}) $ and a
vector $ \varphi \in \mS^{\prime} (\TT^{2}; \RR^{2}) $. In this case we write
\begin{equ}
   (M \rpara \varphi)_{i} = \sum_{j =1}^{2} M_{i, j} \rpara \varphi_{j} \;,
   \qquad (\varphi \para M)_{i} = \sum_{j = 1}^{2} \varphi_{j} \para M_{j, i} \;,
\end{equ}
and similarly all other paraproducts. Note that in this definition $
\varphi \para M$ is \emph{not} the same as $ M \rpara \varphi $. Similarly, for
two matrices $ M, N \in \mS^{\prime} (\TT^{2}; \mathbf{M}^{2}) $ we define
\begin{equ}
   (M \para N)_{i, j}  = \sum_{k = 1}^{2} M_{i, k} \para N_{k, j}\;.
\end{equ}
The following lemma collects the fundamental estimates on paraproducts that we
will make use of: these hold both for vector-valued and matrix-valued
distributions.
\begin{lemma}[Theorems 2.82 and 2.85 \cite{BahouriCheminDanchin2011FourierAndNonLinPDEs}]\label{lem:parap}
  Fix \(\alpha, \beta \in \RR\) and \(p, q \in [1, \infty]\)
  such that \(\frac 1 r = \frac 1 p {+} \frac 1 q \leq 1\). Then uniformly over \(\varphi, \psi \in
  \mS^{\prime}\) 
  \begin{align*}
    \| \varphi \para \psi \|_{\mC^{\alpha}_r} & \lesssim \| \varphi
    \|_{L^{p}} \| \psi \|_{\mC^{\alpha}_{q}} \;, & &  \\ 
    \| \varphi \para \psi \|_{\mC^{\alpha {+} \beta}_{r}} &\lesssim \| \varphi
    \|_{\mC^{\beta}_{p}} \| \psi \|_{\mC^{\alpha}_{q}} \;, \ \ & &  \text{if} \ \ \beta  <
    0 \;, \\
    \| \varphi \reso \psi \|_{\mC^{\alpha {+} \beta}_{r}} & \lesssim \| \varphi
    \|_{\mC^{\beta}_{p}} \| \psi \|_{\mC^{\alpha}_{q}}\;, \ \ & & \text{if} \ \ \alpha
    {+} \beta > 0 \;.
  \end{align*}
\end{lemma}

\end{document}